\documentclass[EJP]{ejpecp}

\usepackage{MnSymbol}
\usepackage{bbm}
\usepackage{listings}
\usepackage{enumerate}
\usepackage{cite}
\usepackage{etoolbox}
\usepackage{ifthen}
\usepackage[boxed]{algorithm2e}

\SHORTTITLE{Consistency thresholds for the planted bisection model}
\TITLE{Consistency thresholds for the planted bisection model}

\AUTHORS{Elchanan Mossel\footnote{
U.C. Berkeley and the University of Pennsylvania. Supported by NSF grants DMS-1106999 and CCF 1320105 and DOD ONR grant N000141110140}
\and Joe Neeman\footnote{
U.T. Austin and the University of Bonn. Supported by NSF grant DMS-1106999 and DOD ONR grant N000141110140}
\and
Allan Sly\footnote{
U.C. Berkeley and the Australian National University. Supported by an
Alfred Sloan Fellowship and NSF grant DMS-1208338.
}}


\AMSSUBJ{05C80} % Edit. Separate items with ;
\SUBMITTED{March 12, 2015}
\ACCEPTED{January 27, 2016}

\ARXIVID{1407.1591}
\VOLUME{21}
\YEAR{2016}
\PAPERNUM{21}
\DOI{16-EJP4185}

\ABSTRACT{
The planted bisection model is a random graph model in which the nodes
are divided into two equal-sized communities and then edges are
added randomly in a way that depends on the community membership.
We establish necessary and sufficient conditions for the asymptotic
recoverability of the planted bisection in this model.
When the bisection is asymptotically recoverable, we
give an efficient algorithm that successfully recovers it.
We also show that the planted bisection is recoverable asymptotically
if and only if with high probability
every node belongs to the same community as the majority of its neighbors.

Our algorithm for finding the planted bisection runs in time almost
linear in the number of edges.
It has three stages: spectral clustering to compute
an initial guess, a ``replica'' stage to get almost every vertex correct, and
then some simple local moves to finish the job.
An independent work by Abbe,  Bandeira, and Hall establishes similar (slightly weaker) results but only in the case of logarithmic average degree.
}

\newcommand{\N}{\mathbb{N}}

\newcommand{\E}{\mathbb{E}}

\newcommand{\calG}{\mathcal{G}}

\newcommand{\calA}{\mathcal{A}}

\newcommand{\normal}{\mathcal{N}}

\newcommand{\pdiffII}[3]{\ifstrequal{#2}{#3}
{\frac{\partial^2 #1}{\partial #2^2}}
{\frac{\partial^2 #1}{\partial #2 \partial #3}}
}
\newcommand{\diffII}[3]{\ifthenelse{\equal{#2}{#3}}
{\frac{d^2 #1}{d #2^2}}
{\frac{d^2 #1}{d #2 d #3}}
}

\newcommand{\eps}{\epsilon}

\let\Pr\relax
\DeclareMathOperator{\Pr}{Pr}

\DeclareMathOperator{\Var}{Var}
\DeclareMathOperator{\Cov}{Cov}

\DeclareMathOperator{\Ber}{Bernoulli}
\DeclareMathOperator{\Binom}{Binom}

\DeclareMathOperator{\HyperGeom}{HyperGeom}

\newcommand{\toP}{\stackrel{P}{\to}}

\newcommand{\eqD}{\stackrel{d}{=}}

\newcommand{\symdiff}{\Delta}

\newcommand{\noti}{\lnot i}

\begin{document}

\section{Introduction}

The ``planted bisection model'' is a random graph model with $2n$ vertices that
are divided into two classes with $n$ vertices each. Edges within the classes
are added to the graph independently with probability $p_n$ each, while edges
between the classes are added with probability $q_n$.
Following Bui et al,~\cite{BCLS:87} who studied a related model,
Dyer and Frieze~\cite{DF:89} introduced the planted bisection model
in order to study the average-case complexity of the
\textsc{Min-Bisection} problem, which asks for a bisection of a graph that cuts
the smallest possible number of edges. This problem is known to be
NP-complete in the worst case~\cite{Karp:72}, but on a random graph model with a
``planted'' small bisection one might hope that it is usually easy.
Indeed, Dyer and Frieze showed that if $p_n = p > q = q_n$ are fixed
as $n \to \infty$ then with high probability the bisection that separates the
two classes is the minimum bisection, and it can be found in expected $O(n^3)$
time.

These models were introduced slightly earlier in the statistics literature~\cite{HLL:83}
(under the name ``stochastic block model'')
in order to study the problem of community detection in random graphs.
Here, the two parts of the bisection are interpreted as latent ``communities'' in a
network, and the goal is to identify them from the observed graph structure.
If $p_n > q_n$,
the maximum a posteriori estimate of the true communities is exactly the same
as the minimum bisection (see the discussion leading to Lemma 4.1),
and so the community detection problem on a stochastic
block model is exactly the same as the \textsc{Min-Bisection}
problem on a planted bisection model; hence, we will use the statistical and computer science
terminologies interchangeably. We note, however, the statistics literature is slightly more
general, in the sense that it often
allows $q_n > p_n$, and sometimes relaxes the problem by allowing the detected communities
to contain some errors.

Our main contribution is a necessary and sufficient condition on
$p_n$ and $q_n$ for recoverability of the planted bisection.
When the bisection can be recovered, we provide an efficient algorithm
for doing so.

\section{Definitions and results}

\begin{definition}[Planted bisection model]
Given $n \in \N$ and $p, q \in [0, 1]$, we define the random $2n$-node
labelled graph $(G, \sigma) \sim \calG(2n, p, q)$ as follows:
first, choose a balanced labelling $\sigma$ uniformly at random
from $\{\tau \in \{1, -1\}^{V(G)}: \sum_u \tau_u = 0\}$.
Then, for every distinct pair $u, v \in V(G)$ independently,
add an edge between $u$ and $v$ with probability $p$ if $\sigma_u = \sigma_v$,
and with probability $q$ if $\sigma_u \ne \sigma_v$.
\end{definition}

The oldest and most fundamental question about planted partition models
is the label reconstruction problem: if we were given the graph $G$ but not
the labelling $\sigma$, could we reconstruct $\sigma$ (up to its sign) from
$G$? This problem is usually framed in the asymptotic regime, where
the number of nodes $n \to \infty$, and $p$ and $q$ are allowed to depend
on $n$.

\begin{definition}[Strong consistency]
 Given sequences $p_n$ and $q_n$ in $[0, 1]$, and given a map $\calA$
 from graphs to vertex labellings, we say that $\calA$ is
 \emph{strongly consistent} (or sometimes just \emph{consistent}) if
 \[
   \Pr_n(\calA(G) = \sigma \text{ or } \calA(G) = -\sigma) \to 1,
 \]
 where the probability $\Pr_n$ is taken with respect to
 $(G, \sigma) \sim \calG(2n, p_n, q_n)$.
\end{definition}

Depending on the application, it may also make sense to ask for a
labelling which is almost completely accurate, in the sense that
it correctly labels all but a vanishingly small fraction of nodes.
Amini et al.~\cite{ACBL:13} suggested the term ``weak consistency'' for this notion.

\begin{definition}[Weak consistency]
 Given $\sigma, \tau \in \{1, -1\}^{2n}$, define
 \[
  \Delta(\sigma, \tau) = 1 - \frac 1{2n} \left|\sum_{i=1}^{2n} \sigma_i \tau_i\right|.
 \]
  Given sequences $p_n$ and $q_n$ in $[0, 1]$, and given a map $\calA$
 from graphs to vertex labellings, we say that $\calA$ is
 \emph{weakly consistent} if
 \[
  \Delta(\sigma, \calA(G)) \toP 0,
 \]
 where ``$\toP$'' means convergence in probability, and
 the probability is taken with respect to
 $(G, \sigma) \sim \calG(2n, p_n, q_n)$.
\end{definition}

Our main result is a characterization of the sequences $p_n$ and $q_n$
for which consistent or weakly consistent estimators exist.
Note that the characterization of weak consistency was obtained previously
by Yun and Proutiere~\cite{YunProutiere:14}, but we include it here
for completeness.

\begin{definition}\label{def:P}
 Given $m$, $n$, $p$, and $q$, let
 \begin{align*}
   X &\sim \Binom(m, \max\{p,q\}) \\
   Y &\sim \Binom(n, \min\{p,q\}).
 \end{align*}
 We define
 \[
  P(m, n, p, q) = \Pr(Y \ge X).
 \]
 When $m = n$, we will abbreviate by $P(n, p, q) = P(n, n, p, q)$.
\end{definition}

\begin{theorem}[Characterization of consistency]\label{thm:consistency}
 Consider sequences $p_n$ and $q_n$ in $[0, 1]$.
 There exists a strongly consistent estimator for
 $\calG(2n, p_n, q_n)$ if and only if $P(n, p_n, q_n) = o(n^{-1})$.
 There exists a weakly consistent estimator for
 $\calG(2n, p_n, q_n)$ if and only if $P(n, p_n, q_n) \to 0$.
\end{theorem}

In order to provide some intuition for Definition~\ref{def:P} and
its appearance in our characterization, we note the following graph-theoretic
interpretation of $P(n, p, q)$:

\begin{definition}\label{def:majority}
 Given a labelled graph $(G, \sigma) \sim \calG(2n, p, q)$ and a node
 $v \in V(G)$, we say that $v$ has a majority of size $k$ if either
 \[
  p > q \text{ and }
  \#\{u \sim v: \sigma_u = \sigma_v\} \ge \#\{u \sim v: \sigma_u \ne \sigma_v\} + k
 \]
 or
 \[
  p < q \text{ and }
  \#\{u \sim v: \sigma_u \ne \sigma_v\} \ge \#\{u \sim v: \sigma_u = \sigma_v\} + k.
 \]
 We say that $v$ has a majority if it has a majority of size one.
 If $v$ does not have a majority, we say that it has a minority.
\end{definition}

\begin{proposition}\label{prop:majorities}
 Fix sequences $p_n$ and $q_n$ in $[0, 1]$ and let $(G, \sigma) \sim
 \calG(n, p_n, q_n)$. Then
 \begin{itemize}
 \item $P(n, p_n, q_n) = o(n^{-1})$ if and only if
 a.a.s.\ every $v \in V(G)$ has a majority; and
 \item $P(n, p_n, q_n) \to 0$ if and only if
 a.a.s.\ at most $o(n)$ nodes in $V(G)$ fail to have a majority.
 \end{itemize}
\end{proposition}

Proposition~\ref{prop:majorities} suggests some intuition
for Theorem~\ref{thm:consistency}: namely, that a node can be labelled
correctly if and only if it has a majority.
In fact, having a majority is necessary for correct labelling
(and we will use this to prove one direction of Theorem~\ref{thm:consistency});
however, it is not sufficient. For example, there are regimes in which
51\% of nodes have majorities, but only 50\% of them can be correctly labelled
(see~\cite{MoNeSl:13}).

We note that Theorem~\ref{thm:consistency} has certain
parallels with local-to-global threshold phenomena in random graphs. For example,
Erd\H{o}s and R\'enyi showed~\cite{ErdosRenyi:61} that for $\calG(n, p_n)$,
if $p_n$ is large enough so that with high probability every node has a neighbor
then the graph is connected with high probability. On the other hand,
every node having a neighbor is clearly necessary for the graph to be
connected. An analogous story holds for the existence of Hamiltonian cycles:
Koml\'os and Szemer\'edi~\cite{KomlosSzemeredi:83} showed that $\calG(n, p_n)$
has a Hamiltonian cycle with high probability if and only if with high probability
every node has degree at least two.

These results on connectedness and Hamiltonicity have a feature in common: in both
cases, an obviously necessary local condition turns out to also be sufficient (on
random graphs) for a global condition.
One can interpret Theorem~\ref{thm:consistency} similarly:
the minimum bisection in $\calG(n, p_n, q_n)$ equals the planted bisection with high
probability
if and only if with high probability every node
has more neighbors of its own label than those of the other label.

\subsection{The algorithm}
In order to prove the positive direction of Theorem~\ref{thm:consistency},
we provide an algorithm that recovers the planted bisection
with high probability whenever $P(n, p_n, q_n) = o(n^{-1})$.
Moreover, this algorithm runs in time $\tilde O(n^2 (p_n + q_n))$, where $\tilde O$ hides
polylogarithmic factors. That is, it runs in time that is almost linear in the number
of edges. In addition, we remark that the algorithm does not need to know $p_n$ and $q_n$.
For simplicity, we assume that we know whether $p_n > q_n$ or vice versa, but this can be checked
easily from the data (for example, by checking the sign of the second-largest-in-absolute-value
eigenvalue of the adjacency matrix; see Section~\ref{sec:spectral}).

Our algorithm comes in three steps, each of which is based on an idea that has already
appeared in the literature. Our first step is a spectral algorithm, along the lines of those
developed by Boppana~\cite{B:87}, McSherry~\cite{M:01}, and Coja-Oghlan~\cite{CO:10}.
Yun and Proutiere~\cite{YunProutiere:14} recently made some improvements to (a special case of) Coja-Oghlan's work,
showing that a spectral algorithm can find a bisection with $o(n)$ errors
if $n \frac{(p_n - q_n)^2}{p_n + q_n} \to \infty$; this is substantially weaker than McSherry's condition for strong consistency, which would
require converging to infinity with a rate of at least $\log n$.

The second stage of our algorithm is to apply a ``replica trick.'' We hold out a small subset $U$ of vertices
and run a spectral algorithm on the subgraph induced by $V \setminus U$. Then we label vertices in $U$ by examining
the edges between $U$ and $V \setminus U$. By repeating the process for many subsets $U$, we dramatically
reduce the number of errors made by the spectral algorithm. More importantly, we get extra information
about the structure of the errors; for example, we can show that the set of incorrectly-labelled vertices
is very poorly connected. Similar ideas are used by Condon and Karp~\cite{CK:01}, who used successive augmentation
to build an initial guess on a subset of vertices, and then used that guess to correctly classify the
remaining vertices. The authors~\cite{MoNeSl:14b} also used a similar idea in the $p_n, q_n = \Theta(n^{-1})$
regime, with a more complicated replica trick based on belief propagation.

The third step of our algorithm is a hill-climbing algorithm, or a sequence of local improvements.
We simply relabel vertices so that they agree with the majority of their neighbors.
An iterative version of this procedure was
considered in~\cite{CI:01}, and a randomized version (based on simulated annealing) was studied by
Jerrum and Sorkin~\cite{JS:98}. Our version has better performance guarantees because we
begin our hill-climbing just below the summit: as we will show, we need to relabel only a tiny fraction
of the vertices and each of those will be relabelled only once.

As noted above, none of the ingredients in our algorithm are novel on their own. However,
the way that we combine them is new (and also crucial to the correctness
of the resulting algorithm). For example, McSherry~\cite{M:01} used a spectral
algorithm with a ``clean-up'' stage, but his clean-up stage was different from our
second and third stages.

\subsection{Formulas in terms of $p_n$ and $q_n$}\label{sec:explicit}
Although Theorem~\ref{thm:consistency} is not particularly explicit
in terms of $p_n$ and $q_n$,
one can obtain various explicit characterizations
in particular regimes (for example, in order to better compare our results
with the existing literature).
We will focus our attention on the case where $p_n$ and $q_n$ are bounded
away from one; for concreteness, suppose $p_n, q_n \le 2/3$. Because of the
symmetry of the problem, this case suffices: 
indeed, replacing $G \sim \calG(n, p_n, q_n)$ by
its complement (the graph in which two vertices are connected if they are
not connected in $G$) corresponds to replacing $p_n$ by $1-p_n$ and $q_n$
by $1-q_n$. Hence, if we handle the case $p_n, q_n \le 2/3$ then we also
handle the case $p_n, q_n \ge 1/3$. There remains the case in which
$\min\{p_n, q_n\} \le 1/3$ and $2/3 \le \max\{p_n, q_n\}$, but this case is
trivial: $P(n, p_n, q_n)$ decreases exponentially fast in $n$, and even very
simple algorithms are known to be strongly consistent.

One can easily see that to obtain strong consistency, at least one of
$p_n$ or $q_n$ must be at least $n^{-1} \log n$ asymptotically. Indeed,
suppose $q_n \le p_n = n^{-1} \log n$ and let $X \sim \Binom(n, p_n)$, $Y \sim \Binom(n, q_n)$. Then $\Pr(X = 0) = \Theta(n^{-1})$, and so
certainly $P(n, p_n, q_n) = \Pr(Y \ge X) = \Omega(n^{-1})$,
which means that strong consistency is impossible for these parameters.
However, strong consistency is possible for some other parameters
in the range $\Theta(n^{-1} \log n)$. Using a Poisson approximation,
we can characterize explicitly which of these sequences allow for
strong consistency:
\begin{proposition}\label{prop:explicit-sparse}
    Let $p_n = a_n n^{-1} \log n$ and $q_n = b_n n^{-1} \log n$.
    If there is a constant $C$ such that $C^{-1} \le a_n, b_n \le C$ for all but
    finitely many $n$ then $P(n, p_n, q_n) = o(n^{-1})$ if and only if
\[
 (a_n + b_n - 2 \sqrt{a_n b_n} - 1) \log n + \frac 12 \log \log n \to
 \infty.
\]
\end{proposition}

In a denser regime, it is tempting to approximate $\Binom(n,p_n)$ and $\Binom(n,q_n)$ by the normal random variables
$\normal(n p_n, n \sigma_p^2)$ and $\normal(n q_n, n \sigma_q^2)$, where
$\sigma_p = \sqrt{p(1-p)}$ and $\sigma_q = \sqrt{q(1-q)}$. That is,
\begin{align*}
\Pr(Y \geq X) &\approx \Pr(\normal(n p_n, n \sigma_p^2) \geq \normal(n q_n, n \sigma_q^2)) \\
&= \Pr(\sigma_p \normal(0,1) \geq \sqrt{n}(q_n - p_n) + \sigma_q \normal(0,1)) \\ &=
\Pr(\normal(0,1) \geq \sigma^{-1} \sqrt{n}(q_n - p_n)),
\end{align*}
where $\sigma = \sqrt{\sigma_p^2 + \sigma_q^2}$. The central limit theorem implies that
the normal approximation is correct in the bulk of the distribution if $np_n \to \infty$ and
$n q_n \to \infty$. However, we are interested in applying this approximation for the tail,
which requires a faster increase of $n p_n$ and a more delicate argument.

\begin{proposition}\label{prop:explicit-dense}
Suppose
$
p_n,q_n = \omega\left( n^{-1} \log^3 n \right)$
and $p_n, q_n \leq 2/3.$
Then the following conditions are equivalent
\begin{itemize}
\item
$P(n, p_n, q_n) = o(1/n)$
\item
$n \Pr\left(\normal(0,1) \geq \sigma_n^{-1} \sqrt{n}(p_n - q_n)\right) \to 0$
\item
$\frac{\sqrt{n} \sigma_n}{p_n-q_n} \exp(-\frac{n (p_n - q_n)^2}{2 \sigma_n^2}) \to 0$,
\end{itemize}
where $\sigma_n = \sqrt{p_n(1-p_n) + q_n(1-q_n)}$.
\end{proposition}
In particular, the third condition in Proposition~\ref{prop:explicit-dense} 
gives an explicit formula for checking whether a strongly consistent estimator
exists.

The formula for weak consistency is rather simpler:
\begin{proposition}\label{prop:weak-consistency}
 $P(n, p_n, q_n) \to 0$ if and only if
 $\frac{n(p_n - q_n)^2}{p_n + q_n} \to \infty$.
\end{proposition}
One direction of Proposition~\ref{prop:weak-consistency} follows from Chebyshev's inequality,
while the other follows from the central limit theorem.

\subsection{Relation to prior work}

Over the years, various authors have
improved on the seminal work of Dyer and Frieze~\cite{DF:89}
by proving weaker sufficient conditions on the sequences $p_n$ and $q_n$
for which the planted bisection
can be recovered. (Various results also generalized the problem by allowing
more than two labels, but we will ignore this generalization here.)
For example, Jerrum and Sorkin~\cite{JS:98}
required $p_n - q_n = \Omega(n^{-1/6 + \epsilon})$, while Condon and Karp
improved this to $p_n - q_n = \Omega(n^{-1/2 + \epsilon})$.
McSherry~\cite{M:01} made a big step by showing that if
\[
\frac{p_n-q_n}{p_n} \geq C \sqrt{\frac{\log n}{p_n n}}
\]
for a large enough constant $C$ then spectral methods can exactly recover the labels.
This was significant because it allowed $p_n$ and $q_n$ to be as small as $\Theta(n^{-1} \log n)$,
which is order-wise the smallest possible.
A similar result for a slightly different random graph model had
been claimed earlier by Boppana~\cite{B:87}, but the proof was incomplete.
Carson and Impagliazzo~\cite{CI:01} showed that with slightly worse poly-logarithmic factors,
a simple hill-climbing algorithm also works.
Analogous results were later obtained by by Bickel and Chen~\cite{BC:09} using
modularity maximization (for which no efficient algorithm is known).

Until now, none of the sufficient conditions in the literature
were also necessary; in fact, necessary conditions on $p_n$ and $q_n$
have only rarely been discussed.
It is instructive to keep the example $p_n=1/2$, $q_n = 1/2-r_n$ in mind. In this case
McSherry's condition is the same as requiring that
$r_n \geq C \sqrt{n^{-1} \log n}$.
On the other hand, Carson and Impagliazzo~\cite{CI:01} pointed out that if $r_n \le c \sqrt{n^{-1} \log n}$ for
some small constant $c$ then the minimum bisection no longer coincides with the planted bisection (as far as we are aware, this was the only necessary
condition in the literature).
From a statistical point of view, this means that the true communities can no longer be reconstructed
perfectly. Our contribution closes the gap between McSherry's sufficient condition and
Carson-Impagliazzo's necessary condition. In the above case, for example,
Proposition~\ref{prop:explicit-dense} shows that the critical
constant is $C = c = 1$.

\subsection{Parallel independent work}

Abbe et al.~\cite{AbBaHa:14} independently
studied the same problem in the logarithmic sparsity regime. They consider
$p_n = (a \log n)/n$ and $q_n = (b \log n) / n$ for constants $a$ and $b$;
they show that $(a+b) - 2\sqrt{ab} > 1$ is
sufficient for strong consistency and that
$(a+b) - 2\sqrt{ab} \ge 1$ is necessary. Note that these are implied
by Proposition~\ref{prop:explicit-sparse}, which is more precise. 
Abbe et al.\ also consider a semidefinite programming algorithm
for recovering the labels; they show that it performs well under
slightly stronger assumptions.

\subsection{Other related work, and an open problem}

Consistency is not the only interesting notion that one can study
on the planted partition model. Earlier work by the
authors~\cite{MoNeSl:13,MoNeSl:14} and by Massouli\'e~\cite{Massoulie:13}
considered a much weaker notion of recovery: they only asked whether
one could find a labelling that was positively correlated with the
true labels.

There are also model-free notions of consistency. Kumar and
Kannan~\cite{KumarKannan:10} considered a deterministic spatial clustering
problem and showed that if every point is substantially closer to the
center of its own cluster than it is to the center of the other cluster
then one can exactly reconstruct the clusters. This is in much the same
spirit as Theorem~\ref{thm:consistency}.

Makarychev, Makarychev, and Vijayaraghavan~\cite{MaMaVi:12,MaMaVi:14}
proposed semi-random models for planted bisections. These models
allow for adversarial noise, and also allow edge distributions that
are not independent, but only invariant under permutations.
They then give approximation algorithms for \textsc{Min-Bisection},
which they prove to work under expansion conditions that hold
with high probability for their semi-random model.

We ask whether the techniques developed here could sharpen the results
obtained by Makarychev et al. For example, exact recovery under adversarial
noise is clearly impossible, but if the adversary is restricted to adding $o(n)$
edges, then maybe one can guarantee almost exact recovery.

\section{Binomial probabilities and graph structure}

In this section, we will prove Proposition~\ref{prop:majorities}, which relates
the binomial probabilities $P(n, p_n, q_n)$ to the structure of
random graphs $G \sim \calG(2n, p_n, q_n)$.

From now on, the letters $c$ and $C$ refer to positive
constants, whose value may change from line to line. We adopt the convention
that $C$ refers to a ``sufficiently large'' constant, so that any statement
involving $C$ will remain true if $C$ is replaced by a larger constant.
Similarly, $c$ refers to a ``sufficiently small'' constant.

\subsection{Binomial perturbation estimates}
We begin by stating some estimates on how binomial probabilities
respond to perturbations, which we will prove in Section~\ref{sec:binomial}.
For example, we will use the following proposition for two
main applications: when $n = m$ and
$\ell = (np)^{1/2} \log^{-1/2} n$, it can be used to
get large majorities ``for free,'' by implying that if
every node has a majority a.a.s., then in fact every node has
a majority of size $(np)^{1/2} \log^{-1/2} n$ a.a.s.
On the other hand, we will also apply Proposition~\ref{prop:large-majority}
with $m = n-1$ and $\ell = 1$, which will be useful (later in this section)
for showing that whether $u$
has a majority is almost independent of whether $v$ has a majority.

\begin{proposition}\label{prop:large-majority}
 Let $X \sim \Binom(m, p)$ and $Y \sim \Binom(n, q)$,
 where $mp \ge 64 \log m$ and $p \le 2/3$. For any $1 \le \ell \le \sqrt{mp \log m}$,
 \begin{align}
  \Pr(Y \ge X + \ell) &\ge
   \Pr(Y \ge X) e^{\big( -C \ell \sqrt{\frac{\log m}{mp}}\big)} - 2m^{-2}
  \label{eq:large-majority-pos} \\
  \Pr(Y \ge X - \ell) &\le
  \Pr(Y \ge X)
  e^{\big(
    C \ell \sqrt{\frac{\log m}{mp}}\big)
  } + 2m^{-2},\label{eq:large-majority-neg}
 \end{align}
 where $C > 0$ is a universal constant.
\end{proposition}

Note that the condition $mp \ge 64 \log m$ is not only a technical one
(although the constant $64$ is certainly not optimal).
For example, if $p = m^{-1} \log m$ and $q = 0$
then~\eqref{eq:large-majority-neg} fails to hold, because
$\Pr(Y \ge X) = \Pr(X = 0) \sim m^{-1}$ but
$\Pr(Y \ge X - 1) = \Pr(X \le 1) \sim m^{-1} \log m$.

Nevertheless, it is still possible to consider similar estimates
in the sparse case. Here is an analogue of~\eqref{eq:large-majority-neg}
that holds with $p = O(m^{-1} \log m)$.

\begin{proposition}\label{prop:large-majority-sparse}
 If $\frac 12 \log m \le mp \le 128 \log m$ and $1 \le \ell \le \log m$
 then
 \[
  \Pr(Y \ge X - \ell) \le \left(\frac{C \log m}{\ell}\right)^{C\ell}
  \Pr(Y \ge X),
 \]
 where $C > 0$ is a universal constant.
\end{proposition}

\subsection{Majorities are uncorrelated}

The preceding propositions may be combined to 
show that the event that $u$ has a minority is essentially independent
of the event that $v$ has a minority.
First, we observe that removing one trial from a binomial random
variable doesn't change very much.

\begin{lemma}\label{lem:one-less}
  There is a universal constant $C > 0$ such that for all
  $m, n$ and all $p, q \le 2/3$,
  \[
  (1 - C m^{-1/3}) P(m-1, n, p, q) - 2m^{-2}
    \le P(m, n, p, q)
    \le (1 + C n^{-1/3}) P(m, n-1, p, q) + 2n^{-2}.
  \]
\end{lemma}

\begin{proof}
  Assume without loss of generality that $p \ge q$.
  Let $X' \sim \Binom(m-1, p)$, $Y' \sim \Binom(n-1, q)$,
  $\xi_X \sim \Ber(p)$ and $\xi_Y \sim \Ber(q)$ be independent,
  and then take $X = X' + \xi_X$ and $Y = Y' + \xi_Y$. In terms
  of these variables, the left-hand inequality above may be written 
  as
  \[
  (1- C m^{-1/3}) \Pr(Y \ge X') - 2m^{-2} \le \Pr(Y \ge X)
  \]
  We will focus on this inequality (since the other inequality
  is essentially identical). Now,
 \begin{align}
  \Pr(Y \ge X) &= \Pr(\xi_X = 0, Y \ge X') + \Pr(\xi_X = 1, Y \ge X' + 1) \notag \\
  &= (1-p) \Pr(Y \ge X') + p \Pr(Y \ge X' + 1).
  \label{eq:one-less-1}
 \end{align}
 If we assume that $(m-1)p \ge 64 \log (m-1)$
 then~\eqref{eq:large-majority-pos} implies that
 \begin{align*}
  p \Pr(Y \ge X' + 1)
  &\ge p\left(1 - C \sqrt\frac{\log m}{mp}\right) \Pr(Y \ge X') - 2m^{-2}\\
  &\ge \left(p - C \sqrt\frac{\log m}{m}\right) \Pr(Y \ge X') - 2m^{-2}.
 \end{align*}
 Plugging this into~\eqref{eq:one-less-1} yields
 \[
 (1 - C m^{-1/3}) \Pr(Y \ge X') - 2m^{-2} \le \Pr(Y \ge X),
 \]
 which implies the claim. On the other hand, if $(m-1)p \le 64 \log (m-1)$ then
 directly from~\eqref{eq:one-less-1} we have
 \[
  \Pr(Y \ge X) \ge (1-p) \Pr(Y \ge X') \ge (1 - C m^{-1/3}) \Pr(Y \ge X').
  \qedhere
 \]
\end{proof}

Next, we show that $\{u \text{ has a minority}\}$ and
$\{v \text{ has a minority}\}$ are essentially uncorrelated.
We recall that if $A$ and $B$ are events then
$\Cov(A, B) = \Pr(A \cap B) - \Pr(A) \Pr(B)$.

\begin{lemma}\label{lem:minority-covariance}
Fix nodes $u$ and $v$. Let $A$ and $B$
be the events that $u$ and $v$ respectively have minorities.
If $p, q \le 2/3$ then
\[
  |\Cov(A, B)| \le C n^{-1/3} \Pr(A) \Pr(B) + C n^{-4}.
\]
\end{lemma}

\begin{proof}
 Assume that $p > q$ and that $\sigma_u = +$ and $\sigma_v = -$
 (the other cases are very similar).
 Let $\xi$ be the indicator that $u \sim v$, and let $A$ and $B$
 be the events that $u$ and $v$ respectively have minorities.
 Note that $A$ and $B$ are conditionally independent given $\xi$,
 which means that
 \begin{align*}
  \Cov(A, B)
  &= \Cov(\Pr(A \mid \xi), \Pr(B \mid \xi)) \\
  &\le \sqrt{\Var(\Pr(A \mid \xi)) \Var(\Pr(B \mid \xi))} \\
  &= \Var(\Pr(A \mid \xi)),
 \end{align*}
 where the last equality holds because $A$ and $B$ have the same
 distribution given $\xi$.

 Define
 $\alpha = P(n-1, n, p, q) = \Pr(\text{$u$ has a minority}) =
 \Pr(\text{$v$ has a minority})$. By our assumption that
 $\sigma_u \ne \sigma_v$ and $p > q$, we have
 $\Pr(A \mid \xi = 0) \le \Pr(A \mid \xi = 1)$.
 On the other hand,
 \[
   \Pr(A \mid \xi = 0) = P(n-1, n-1, p, q) \ge (1 - C n^{-1/3}) \alpha - 2n^{-2}.
 \]
 by Lemma~\ref{lem:one-less}.

 Next, we consider $\Pr(A \mid \xi = 1)$. Note that
 \begin{align*}
  \Pr(A \mid \xi = 1)
  &= \Pr(1 + \Binom(n-1, q) \ge \Binom(n-1, p)) \notag \\
  &\le \Pr(1 + \Binom(n, q) \ge \Binom(n-1, p)).
   %\label{eq:minority-covariance-2}
 \end{align*}
 By applying either~\eqref{eq:large-majority-neg}
 or Proposition~\ref{prop:large-majority-sparse}
 to the right hand side above, we have
 \[
   \Pr(A \mid \xi = 1) \le \begin{cases}
     (1 + C n^{-1/6}) \alpha + 2 n^{-2} & \text{$p \ge n^{-1/2}$} \\
     \alpha \log^C n + 2 n^{-2} & \text{otherwise}.
   \end{cases}
 \]
 (To get the second case, we are either applying~\eqref{eq:large-majority-neg}
 for $64 \log n \le np \le n^{1/2}$ or we are applying
 Proposition~\ref{prop:large-majority-sparse}.)
 In the first case, the random variable $\Pr(A \mid \xi)$
 is supported on an interval of width at most
 $C n^{-1/6} \alpha + C n^{-2}$ and so its variance is at most
 $C n^{-1/3} \alpha^2 + C n^{-4}$.
 In the second case, $\Pr(\xi = 1) = q \le p \le n^{-1/2}$,
 and so
 \begin{align*}
   \Var(\Pr(A \mid \xi))
   &\le \E (\Pr(A \mid \xi) - \alpha)^2 \\
   &\le \Pr(\xi = 0) C \alpha^ 2n^{-2/3} + \Pr(\xi = 1) C \alpha^2 \log^{2C} n + Cn^{-4},
 \end{align*}
 which is bounded by $C \alpha^2 n^{-1/3} + Cn^{-4}$.
\end{proof}

\subsection{Graph structure}
 Finally, we will use our preceding estimates to prove
 Proposition~\ref{prop:majorities}. Most of the proof essentially follows
 by straightforward first moment arguments. The most complicated part
 is showing that $P(n, p_n, q_n) = \Omega(n^{-1})$ implies that with constant probability there exists a
 node with a minority. This uses a fairly standard second
 moment argument,
 the main technical part of which is contained in
 Lemma~\ref{lem:minority-covariance}.

\begin{proof}[Proof of Proposition~\ref{prop:majorities}]
  Fix a node $v \in V(G)$ and suppose without loss of generality
  that $\sigma_v = +$. For notational convenience, we will also suppose
  that $p > q$; an essentially identical proof works for $p < q$.
  Let $X$ and $Y$ denote the number of $+$- and $-$-labelled neighbors
  of $v$.
  Then
  \begin{align*}
   X &\sim \Binom(n-1, p_n) \\
   Y &\sim \Binom(n, q_n).
  \end{align*}
  Suppose first that $P(n, p_n, q_n) = o(1)$. Then
  \[
   \Pr(\text{$v$ has a minority}) = \Pr(Y \ge X) = P(n-1, n, p_n, q_n) = o(1)
  \]
  by Lemma~\ref{lem:one-less}.
  Summing over $v \in V(G)$, we have
  \[
   \E(\text{\# of nodes with a minority}) = o(n),
  \]
  and so Markov's inequality implies that a.a.s.\ all but $o(n)$ nodes
  have a majority.

  The case where $P(n, p_n, q_n) = o(n^{-1})$ is very similar, except that
  we conclude with
  $\E(\text{\# of nodes with a minority}) = o(1)$,
  which implies that a.a.s.\ every node has a majority.

  For the rest of the proof, we will assume that $p_n, q_n \le 2/3$. As
  we explained in Section~\ref{sec:explicit}, this case suffices: if
  $p_n, q_n \ge 1/3$ then we may apply the result with $p_n$
  and $q_n$ replaced by $1-p_n$ and $1 - q_n$; if $q_n \le 1/3$ and
  $p_n \ge 2/3$ then $P(n, p_n, q_n) = o(n^{-1})$ and we have already
  given that part of the proof.

  Suppose that the number of nodes without a majority is not $o(n)$ a.a.s.
  Then there is some $\epsilon > 0$ such that for infinitely many $n$,
  the probability of having $\epsilon n$ nodes with a minority is at least
  $\epsilon$. Thus, the expected number of nodes with a minority is at least
  $\epsilon^2 n$ for infinitely many $n$, which in turn implies that
  $P(n-1, n, p_n, q_n) = \Pr(Y \ge X) \ge \epsilon^2$ for infinitely many $n$. By
  Lemma~\ref{lem:one-less}, $P(n, p_n, q_n) \not \to 0$.

  It remains to prove that all nodes have a majority a.a.s.\ only if
  $P(n, p_n, q_n) = o(n^{-1})$. This requires a second moment argument:
  let $\xi_u$ be the indicator that $u$ has a minority
  and let $N = \sum_u \xi_u$ be the number of nodes with a minority. If
  $\alpha = \Pr(\text{$u$ has a minority})$ (which is the same for all $u$) then
  \begin{align*}
   \Var(N)
   &= \sum_u \Var(\xi_u) + \sum_{u \ne v} \Cov(\xi_u, \xi_v) \\
   &\le n \alpha + C n^2 \alpha^2 n^{-1/3} + C n^{-2},
  \end{align*}
  where the last line follows from Lemma~\ref{lem:minority-covariance}.
  In particular, we may bound $\Var(N) \le C \max\{\E N, (\E N)^2, n^{-2}\}$. Now, if
  $P(n, p_n, q_n)$ is not $o(n^{-1})$ then there is some $\epsilon > 0$
  and infinitely many $N$ for which $\E N \ge \epsilon$. By
  the Paley-Zygmund inequality and our bound on $\Var(N)$,
  there is some $\delta > 0$ such that
  for infinitely many $n$, $\Pr(N \ge \delta) \ge \delta$.
  Since $\{N > 0\} = \{\exists u \text{ with a minority}\}$, this implies that the event of having
  only majorities is not asymptotically almost sure.
\end{proof}

\section{Sufficient condition for strong consistency}\label{sec:sufficient}

The rough idea behind our strongly consistent labelling algorithm is to
first run a weakly consistent algorithm and then try to improve it.
The natural way to improve an almost-accurate labelling $\tau$
is to search for nodes $u$ that have a minority with respect to $\tau$
and flip their signs. In fact, if the errors in $\tau$ were independent
of the neighbors of $u$ then this would work quite well:
assuming that $u$ has a decently large majority (which it will, for
most $u$, by Proposition~\ref{prop:large-majority}), then having a
labelling $\tau$ with few errors is like observing each neighbor of $u$
with a tiny amount of noise. This tiny amount of noise is very unlikely
to flip $u$'s neighborhood from a majority to a minority. Therefore,
choosing $u$'s sign to give it a majority is a reasonable approach.

There are two important problems with the argument outlined in the
previous paragraph: it requires the errors in $\tau$
to be independent, and it is only guaranteed to work
for those $u$ that have a sizeable majority (i.e., almost, but not quite,
all the nodes in $G$). Nevertheless, this procedure is a good starting
point and it motivates the first clean-up stage of our
algorithm (Algorithm~\ref{alg:first}). By removing
$u$ from the graph before looking for the almost-accurate labelling $\tau$,
we ensure the required independence properties
(as a result, note that we will be dealing with multiple labellings $\tau$,
depending on which nodes we removed before running our almost-accurate labelling
algorithm). And although the final
labelling we obtain is not guaranteed to be entirely correct, we show
that it has very few (i.e., at most $n^\epsilon$) errors whereas
the initial labelling was only guaranteed to have $o(n)$ errors.

In order to finally produce the correct labelling, we
return to the earlier idea: flipping the label of every node
that has a minority. We analyze this procedure by noting that
after the previous step of the algorithm, the errors were confined to
a very particular set of nodes (namely, those without a very strong majority).
We show that this set of nodes is small and poorly connected, which
means that every node in the graph is guaranteed to only have a few neighbors
in this bad set. In particular, even nodes with relatively weak majorities
cannot be flipped by labelling errors in the bad set.
We analyze this procedure in Section~\ref{sec:second-step}.

\subsection{The initial guess}\label{sec:spectral}

As stated in the introduction,
there exist algorithms for a.a.s.\ correctly labelling all but
$o(n)$ nodes. Assuming that $p_n + q_n = \Omega(n^{-1} \log n)$, such
an algorithm is easy to describe, and we include it for completeness;
indeed, the algorithm we give is essentially
folklore, although a nice treatment is given in~\cite{NadakuditiNewman:12}.
A slightly more complex algorithm that doesn't assume
$p_n + q_n = \Omega(n^{-1} \log n)$ can be found in~\cite{YunProutiere:14}.

Note that the conditional expectation
of the adjacency matrix given the labels is $\frac{p_n + q_n}{2} 1 1^T + \frac{p_n - q_n}{2} \sigma \sigma^T$,
where $\sigma \in \{\pm 1\}^{2n}$ is the true vector of class labels.
Now, let $A$ be the adjacency matrix of $G$. Then $\sigma$ is the second eigenvector of $\E [A \mid \sigma]$,
and its eigenvalue is $\frac{p_n - q_n}{2}$. In particular, if we had access to $\E [A \mid \sigma]$ then
we could recover the labels exactly, simply by looking at its second eigenvector.
Instead, we have access only to $A$. However, if $A$ and $\E[A \mid \sigma]$ are close then
we can recover the labels by rounding the second eigenvector of $A$.

Conditioned on $\sigma$, $A - \E[A \mid \sigma]$ is a
symmetric matrix whose upper triangular part consists of independent entries, and so we can use results from random matrix
theory~\cite{Vu:07,Seginer:00} to bound its norm:
\begin{theorem}\label{thm:norm}
 If $p_n + q_n = \Omega(n^{-1} \log n)$ then there is a constant $C$ such that
 \[
  \|A - \E[A \mid \sigma]\| \le C \sqrt{n (p_n + q_n)}
 \]
 a.a.s.\ as $n \to \infty$, where $\|\cdot\|$ denotes the spectral
 norm.
\end{theorem}

Assuming Theorem~\ref{thm:norm}, note that if
$|p_n - q_n|/\sqrt{n (p_n + q_n)} \to \infty$ then $\|A - \E[A \mid \sigma]\|$ is order-wise
smaller than the second eigenvalue of $A$. By the Davis-Kahan theorem, it is possible to
recover $\sigma$ up to an error of size $o(1) \|\sigma\|$. This implies that we can recover the labels
of all but $o(n)$ vertices.

\subsection{The replica step}\label{sec:first-step}

Let {\tt BBPartition} be an algorithm that is guaranteed to a.a.s.\ label all but $o(n)$ nodes correctly;
we will use it as a black box.
Note that we may assume that {\tt BBPartition} produces an exactly
balanced labelling. If not, then if its output has more $+$ labels
than $-$ labels, say, we can randomly choose some $+$-labelled vertices
and relabel them. The new labelling is balanced, and it is
still guaranteed to have at most $o(n)$ mistakes.

\begin{figure}
\begin{algorithm}[H]
 \SetKwInOut{Input}{input}
 \SetKwInOut{Output}{output}
 \LinesNumbered

 \Input{graph $G$, parameter $\epsilon > 0$}
 \Output{a partition $W_+, W_-$ of $V(G)$}
 \BlankLine

 $W_+ \leftarrow \emptyset$\;
 $W_- \leftarrow \emptyset$\;
 choose $m \in \N$ so $(1-2/m) \epsilon - 80 m^{-1/2} \ge \epsilon/2$\;
 partition $V(G)$ randomly into $U_1, \dots, U_m$\;
 $U_+, U_- \leftarrow \mathtt{BBPartitition}(G)$\;
 \label{alg:first-initial-partition}
 \For{$i \leftarrow 1$ \KwTo $m$} {
  $U_{i,+}, U_{i,-} \leftarrow \mathtt{BBPartition}(G \setminus U_i)$\;
    \If{$|U_{i,+} \symdiff U_+| \ge n/2$ \label{alg:first-before-align}}
     {swap $U_{i,+}$ and $U_{i,-}$\; \label{alg:first-align} }
  \For{$v \in U_i$}{
    \uIf{$p > q$ and $\#\{u \in U_{i,+}: u \sim v\} > \#\{u \in U_{i,-}: u \sim v\}$}{
    \label{alg:first-label}
      $W_+ \leftarrow W_+ \cup \{v\}$\;
    }
    \uElseIf{$p < q$ and $\#\{u \in U_{i,+}: u \sim v\} < \#\{u \in U_{i,-}: u \sim v\}$}{
      $W_+ \leftarrow W_+ \cup \{v\}$\;
    }
    \Else{
      $W_- \leftarrow W_- \cup \{v\}$\;
    }
    \label{alg:first-label-end}
  }
 }
\caption{Algorithm for initial accuracy boost}
\label{alg:first}
\end{algorithm}
\end{figure}

For the remainder of Section~\ref{sec:sufficient}, we will assume
that $p \ge q$ in order to lighten our notation. The case $p < q$
is very similar, except that expressions like $\Pr(Y \ge X - \ell)$ should
be replaced by $\Pr(X \ge Y - \ell)$. We will also assume that $p \le 2/3$;
as discussed in Section~\ref{sec:explicit}, all interesting cases may be
reduced to this one.

We define $V_\epsilon$ to be a set of ``bad'' nodes
that our first step is not required to label correctly.
\begin{definition}
 Let $V_\epsilon$ be the elements of $V$ that have a majority of
 size less than $\epsilon \sqrt{n p \log n}$, or that have more than
 $100np$ neighbors.
\end{definition}

\begin{proposition}\label{prop:first-step}
  For any $\epsilon > 0$,
 Algorithm~\ref{alg:first} a.a.s.\ correctly labels every node
 in $V \setminus V_\epsilon$.
\end{proposition}

Before proving Proposition~\ref{prop:first-step}, we deal with a minor
technical point.
The following lemma shows that we can apply {\tt BBPartition} to subgraphs
of $G \sim \calG(2n, p_n, q_n)$, and it will still have the required guarantees.

\begin{lemma}
 If $P(n, p_n, q_n) = o(n^{-1})$ then for any $\alpha > 0$,
 $P(\lfloor \alpha n \rfloor, p_n, q_n) \to 0$.
\end{lemma}

\begin{proof}
This follows from two simple properties of the function $P$.
First, we have $P(n_1 + n_2, p, q) \ge P(n_1, p, q) P(n_2, p, q)$ for
any $n_1, n_2, p$, and $q$. Indeed, if $X_i \sim \Binom(n_i, p)$
and $Y_i \sim \Binom(n_i, q)$ are independent then
\begin{align*}
 P(n_1 + n_2, p, q)
 &= \Pr(X_1 + X_2 \le Y_1 + Y_2) \\
 &\ge \Pr(X_1 \le Y_1) \Pr(X_2 \le Y_2) \\
 &= P(n_1, p, q) P(n_2, p, q).
\end{align*}

A similar coupling argument shows that for any $n_2 \ge 0$,
$P(n_1, p, q) \ge \frac 12 P(n_1 + n_2, p, q)$. Indeed, conditioned on
$X_1 + X_2 \le Y_1 + Y_2$, the probability of $X_1 \le Y_1$ is at least $\frac 12$.
Hence,
\begin{align*}
P(n_1, p, q)
&= \Pr(X_1 \le Y_1) \\
&\ge \Pr(X_1 \le Y_1 \mid X_1 + X_2 \le Y_1 + Y_2) \Pr(X_1 + X_2 \le Y_1 + Y_2) \\
&\ge \frac 12 P(n_1 + n_2, p, q).
\end{align*}

Now, choose an integer $k$ so that $\alpha \ge 1/k$. Then
\[
 P(n, p, q) \ge \frac 12 P(2k \lfloor n/k\rfloor, p, q)
 \ge \frac 12 P(\lfloor n/k \rfloor, p, q)^{2k}
 \ge \frac 14 P(\lfloor \alpha n \rfloor, p, q)^{2k}.
\]
Since $k$ and $\alpha$ are constant as $n \to \infty$, this completes the proof.
\end{proof}

\begin{proof}[Proof of Proposition~\ref{prop:first-step}]
 First, we may assume without loss of generality that
 the partition $U_+, U_-$ that was produced in
 line~\ref{alg:first-initial-partition} is positively correlated with the true labelling
 $\sigma$.
 By our assumption on \texttt{BBPartition}, at line~\ref{alg:first-before-align} $U_{i,+}$ either agrees with $V_+ \setminus U_i$ or $V_- \setminus U_i$,
 up to an error of $o(n)$. After the relabelling in
 line~\ref{alg:first-align}, then, a.a.s.\ $U_{i,+}$
 agrees with $V_+ \setminus U_i$ up to an error of $o(n)$.
 Since $m$ is a constant independent of $n$, this property a.a.s.\ holds
 for every $i$ simultaneously.

 Now, consider a node $v \not \in V_\epsilon$ and suppose
 without loss of generality that $\sigma_v = +$. Conditioned on $v \in U_i$,
 every other node is added to $U_i$ independently with probability $1/m$.
 Hence, conditioned on $v$ having $k_+$ $+$-labelled neighbors and $k_-$
 $-$-labelled neighbors,
 it has $\Binom(k_+, 1/m)$ $+$-labelled neighbors in $U_i$
 and $\Binom(k_-, 1/m)$ $-$-labelled neighbors in $U_i$.
 Let $k_{+,i}$ denote the number of $+$-labelled neighbors that $v$
 has in $U_i$ and let $k_{+,\noti} = k_+ - k_{+,i}$
 be the number of $+$-labelled neighbors that $v$ has in
 $V \setminus U_i$ (and similarly for $-$).

 By Bernstein's inequality, with probability at least $1-2n^{-2}$,
 \begin{align}
  \label{eq:same-neighbors-1}
  k_{+,i} &\in k_+/m \pm 4 \sqrt{k_+ m^{-1} \log k_+} \\
  \label{eq:same-neighbors-2}
  k_{-,i} &\in k_-/m \pm 4 \sqrt{k_- m^{-1} \log k_-}.
 \end{align}
 Recall that $v \not \in V_\epsilon$ implies that $k_+ \le 100np$,
 $k_- \le 100np$ and
 \[k_+ - k_- \ge \epsilon\sqrt{np\log n}.\]
 Hence,~\eqref{eq:same-neighbors-1} and~\eqref{eq:same-neighbors-2} imply that
 \begin{align*}
   k_{+,\noti} - k_{-,\noti}
   &\ge (1-2/m) \epsilon\sqrt{np \log n} - 4 \sqrt{k_+m^{-1} \log k_+} - 4 \sqrt{k_- m^{-1} \log k_-} \\
  &\ge (1-2/m) \epsilon\sqrt{np \log n} - 80 m^{-1/2} \sqrt{np \log n} \\
  &\ge \frac{\epsilon}{2} \sqrt{np \log n},
 \end{align*}
 where the last inequality follows from the definition of $m$.
 Taking a union bound over the events leading to~\eqref{eq:same-neighbors-1},
 we see that a.a.s., for every $v \not \in V_\epsilon$ with $\sigma_v = +$,
 if $v \in U_i$ then
 \begin{equation}\label{eq:first-step-large-mag}
   (k_{+,\noti} - k_{-,\noti})
  \ge \frac{\epsilon}{2} \sqrt{np \log n}.
 \end{equation}
 In other words, every $v \not \in V_\epsilon$ still has a strong majority,
 even if we consider only edges between $v$ and the complement of $U_i$.

 Let $X_-$ be the number of $+$-valued neighbors of $v$ that were incorrectly
 labelled as $-$ in line~\ref{alg:first-align}
 (i.e. $X_- = |\{u : u \sim v, \sigma_u = +, u \in U_{i,-}\}|$), and let $X_+$ be the number
 of $-$-valued neighbors that were incorrectly labelled as $+$.
 Note that the quantities considered in line~\ref{alg:first-label}
 of Algorithm~\ref{alg:first} may be expressed in terms of
 $k$ and $X$ as
 \begin{align*}
   \#\{u \in U_{i,+}: u \sim v\} &= k_{+,\noti} - X_- + X_+ \\
   \#\{u \in U_{i,-}: u \sim v\} &= k_{-,\noti} + X_- - X_+.
 \end{align*}
 Hence, the inequality $|X_+ - X_-| < \frac 12 |k_{+,\noti} - k_{i,\noti}|$
 will imply that $v$ is correctly labelled in
 lines~\ref{alg:first-label}--\ref{alg:first-label-end}.
 For the rest of the proof, our goal will be to show that a.a.s.\ the
 above inequality holds for all $v \not \in V_\epsilon$.

 Let $E_- = \# \{u \in U_{i,-} : \sigma_u = +\}$
 (i.e., the total number of $+$-labelled vertices
 that were mislabelled in line~\ref{alg:first-align}) and let
 $E_+ = \# \{u \in U_{i,+} : \sigma_u = -\}$.
 Note that the neighbors of $v$ are independent of $U_{i,-}$, and so
 conditioned on $k_{+,\noti}$ and $k_{-,\noti}$,
 \begin{align*}
  X_- &\eqD \HyperGeom(|V_+ \setminus U_i|, k_{+,\noti}, E_-) \\
  X_+ &\eqD \HyperGeom(|V_- \setminus U_i|, k_{-,\noti}, E_+),
 \end{align*}
 where $V_+$ and $V_-$ are the set of $u$ with $\sigma_u = +$
 and $\sigma_u = -$, respectively.
 Now condition on $k_{+,\noti}$ and $k_{-,\noti}$,
 and on the following a.a.s.\ events:
 \begin{gather*}
   \forall i \quad |V_+ \setminus U_i| \in n(1-1/m) \pm \sqrt{n} \log \log n \\
   \forall i \quad |V_- \setminus U_i| \in n(1-1/m) \pm \sqrt{n} \log \log n \\
   |E_- - E_+| \le \sqrt {n \log \log n}.
 \end{gather*}
 Under the above events, and recalling that $k_+ \le 100np$,
 \begin{align*}
  |\E X_- - \E X_+|
  &= \left|
  E_- \frac{k_{+,\noti}}{|V_+ \setminus U_i|}
  - E_+ \frac{k_{-,\noti}}{|V_- \setminus U_i|}
  \right| \\
  &\le \left|
  E_- \frac{k_{+,\noti}}{n(1-1/m)}
  - E_+ \frac{k_{-,\noti}}{n(1-1/m)}
  \right|
  + O(n^{1/2} p \log \log n) \\
  &\le O(n^{-1}) |E_- - E_+| k_+
     + O(n^{-1}) E_+ \big|k_{+,\noti} - k_{-,\noti}\big|
     + O(\sqrt{np} \log \log n) \\
  &\le O(\sqrt{np} \log \log n)
     + o(1) \big|k_{+,\noti} - k_{-,\noti}\big|,
 \end{align*}
 Going back to~\eqref{eq:first-step-large-mag},
 we see that a.a.s.\ for all $v \not \in V_\epsilon$,
 \[
   |\E X_- - \E X_+| \le \frac{1}{8} |k_{+,\noti} - k_{-,\noti}|.
 \]

 Next, we consider the deviations of $X_-$ and $X_+$ around their means.
 By Bernstein's inequality for hypergeometric variables, there is a constant
 $C$ such that with probability $1-n^{-2}$, $X_-$ is within
 \[
  C \sqrt{E_- \frac{k_{+,\noti}}{|V_+ \setminus U_i|} \log E_-}
  \le C' \sqrt{E_- p \log n}
 \]
 of its expectation. Since $E_- = o(n)$, we can take $n$ large enough
 so that $X_-$ is within $\frac{\epsilon}{16}\sqrt{np \log n}$ of its expectation
 with probability $1 - n^{-2}$. Arguing similarly for $X_+$ we have
 \begin{align*}
  |X_- - X_+|
  &\le |\E X_- - \E X_+| + |X_- - \E X_-| + |X_+ - \E X_+| \\
  &\le \frac{1}{8} |k_{+,\noti} - k_{-,\noti}| + \frac{\epsilon}{8} \sqrt{np \log n}
 \end{align*}
 with probability $1-2n^{-2}$. Taking a union bound over $v \not \in V_\epsilon$ (recall that $X$ and $k$ both depend on $v$), we see that the above
 inequality holds a.a.s.\ for all $v \not \in V_\epsilon$ simultaneously.
 By~\eqref{eq:first-step-large-mag}, a.a.s.\ for all $v \in V_\epsilon$,
 \[
   |X_- - X_+| \le \frac 38 |k_{+,\noti} - k_{-,\noti}|,
 \]
 which completes the proof.
\end{proof}

\subsection{The hill-climbing step}\label{sec:second-step}

After running Algorithm~\ref{alg:first}, we are left with a graph
in which only nodes belonging to $V_\epsilon$ could possibly be mis-labelled.
Fortunately, very few nodes belong to $V_\epsilon$, and those that do are poorly
connected to the rest of the graph. This is the content of the next
two propositions.

\begin{proposition}\label{prop:Ve-size}
 For every $\delta > 0$ there exists an $\epsilon > 0$ such that
 if $P(n, p, q) = o(n^{-1})$ then
 $|V_\epsilon| \le n^\delta$ a.a.s.
\end{proposition}

\begin{proof}
 Consider a single $v \in V$. By Bernstein's inequality
 the probability that $v$ has $100np$ neighbors is less than $n^{-2}$
 (using $np \ge \log n$, which follows from $P(n, p, q) = o(n^{-1})$).
 Hence, a.a.s.\ every $v$ has at most $100np$ neighbors.

 It remains to show that a.a.s.\ at most $n^\delta$ vertices
 fail to have a majority of size $\epsilon \sqrt{np \log n}$.
 Now, if $np \ge 64 \log n$ then Proposition~\ref{prop:large-majority}
 with $\ell = \epsilon \sqrt{np \log n}$ implies that
 if $Y \sim \Binom(n, q)$ and $X \sim \Binom(n-1, p)$ then
 \[
  \Pr(Y \ge X - \epsilon \sqrt{np \log n}) \le 2n^{-2} + O(n^{-1 + C \epsilon}).
 \]
 In particular, if $C\epsilon < \delta$ then the right hand size is
 $o(n^{-1 + \delta})$. By Markov's inequality, this implies that a.a.s.\ at most
 $n^{\delta}$ nodes fail to have a majority of size
 $\epsilon \sqrt{np \log n}$.

 In the sparse case (i.e. $\frac 12 \log n \le np \le 128 \log n$),
 Proposition~\ref{prop:large-majority-sparse} with
 $\ell = \epsilon \sqrt{np \log n} = \Theta(\epsilon \log n)$ yields
 \[
  \Pr(Y \ge X - \epsilon \sqrt{np \log n})
  \le (2C/\epsilon)^{C \epsilon \log n} n^{-1}.
 \]
 Since $(2/\epsilon)^{\epsilon} \to 1$ as $\epsilon \to 0$, we may choose
 $\epsilon$ so that $(2C/\epsilon)^{C \epsilon \log n} \le n^{\delta/2}$.
 By Markov's inequality, we see that at most $n^\delta$ nodes
 fail to have a majority of size $\epsilon \sqrt{np \log n}$.
\end{proof}

\begin{proposition}\label{prop:Ve-neighbors}
  Suppose that $P(n, p, q) = o(n^{-1})$ and $np \le n^{1/4}$.
  For sufficiently small $\epsilon$, a.a.s.\ no node has two or more neighbors in $V_\epsilon$.
\end{proposition}

\begin{proof}
  Fix $u, v \in V$; let $X \sim \Binom(n-1, p)$
  and $Y \sim \Binom(n,q)$. As in the proof of
 Proposition~\ref{prop:Ve-size}, a.a.s.\ every $v \in V$ has at most
 $100np$ neighbors; for the rest of the proof, we condition on this
 event. Moreover, we may choose $\epsilon$
 small enough so that $\Pr(Y \ge X - \epsilon \sqrt{np \log n}) \le n^{-7/8}$.
 In particular, that means that $\Pr(u \in V_\epsilon) \le n^{-7/8}$. Now
 condition on the neighbors of $u$. If $v$ has a majority of
 $2\epsilon \sqrt{np \log n}$ on all edges except for $u$, then it lies outside
 of $V_\epsilon$ regardless of whether it neighbors $u$. But this event
 is independent of whether $u \in V_\epsilon$, and if $\epsilon$ is sufficiently
 small then it has probability
 at least $1-n^{-7/8}$. Hence, $\Pr(u, v \in V_\epsilon) \le n^{-7/4}$.

 Now condition on the event that $u, v \in V_\epsilon$.
 Recall that $u$ and $v$ each
 have at most $100np \le 100n^{1/4}$ neighbors in $V_-$ and at most
 $100n^{1/4}$ neighbors in $V_+$. Conditioned on the number of neighbors in $V_-$
 and $V_+$, the neighbors of $u$ and $v$ are independent and uniformly distributed.
 Hence, the probability that they have a common neighbor is $O(n^{-3/4 - 3/4 + 1})
 = O(n^{-1/2})$. Combining this with the previous paragraph, we have
 \[
  \Pr(u, v \in V_\epsilon \text{ and they have a common neighbor})
  = O(n^{-9/4}).
 \]
 Taking a union bound over $n^2$ choices of $u$ and $v$ completes the proof.
\end{proof}

\begin{proposition}\label{prop:Ve-adjacent}
 Suppose that $np \le n^{1/4}$. For sufficiently small $\epsilon$,
 a.a.s.\ no two nodes in $V_\epsilon$ are adjacent.
\end{proposition}

\begin{proof}
 Fix $u, v \in V$. The probability that they are adjacent
 is at most $p \le n^{-3/4}$.
 As in the previous proof, if $\epsilon$ is small enough
 then $\Pr(u \in V_\epsilon \mid u \sim v)$
 and $\Pr(v \in V_\epsilon \mid u \sim v, u \in V_\epsilon)$
 are both at most $n^{-7/8}$.
 Multiplying these conditional probabilities, we have
 \[
  \Pr(u, v \in V_\epsilon \text{ and } u \sim v)
  = O(n^{-5/2}),
 \]
 and we conclude by taking a union bound over $u$ and $v$.
\end{proof}

\begin{algorithm}
 \SetKwInOut{Input}{input}
 \SetKwInOut{Output}{output}
 \SetKwFor{Repeat}{repeat}{do}{end}
 \LinesNumbered

 \Input{graph $G$, an initial partition $U_+, U_-$ of $V(G)$}
 \Output{a partition $W_+, W_-$ of $V(G)$}
 \BlankLine

  $W_+ \leftarrow \{v \in V(G): \text{$v$ has more neighbors in $U_+$ than in $U_-$}\}$\;
  $W_- \leftarrow V(G) \setminus W_+$\;

\caption{Algorithm for final labelling}
\label{alg:second}
\end{algorithm}

\begin{proposition}\label{prop:alg-works}
 Suppose that we initialize Algorithm~\ref{alg:second} with a partition whose
 errors are restricted to $V_\epsilon$, and suppose that $P(n, p_n, q_n) = o(n^{-1})$.
 Then a.a.s., Algorithm~\ref{alg:second} returns the true partition.
\end{proposition}

\begin{proof}
  We consider two cases: the dense regime $n^{1/4} \le np \le 2n/3$, and the sparse
 regime $\frac 12 \log n \le np n^{1/4}$.

 In the dense regime, note that by Proposition~\ref{prop:large-majority},
 a.a.s.\ every node has a majority of $\Omega(\sqrt{np/\log n}) \ge \Omega(n^{1/9})$.
 On the other hand, if $\epsilon$ is sufficiently small then
 (by Proposition~\ref{prop:Ve-size}) $|V_\epsilon| \le n^{1/10}$, which implies that
 every node in $V_+$ will have most of its neighbors in $U_+$.
 Therefore, $W_+ = V_+$ in Algorithm~\ref{alg:second}.

 In the sparse regime, let $V'$ be the set
 of nodes with a majority of less than three; note that $V' \subset V_\epsilon$.
 By Proposition~\ref{prop:Ve-neighbors},
 a.a.s.\ every node has at most one neighbor in $V_\epsilon$, which
 implies that every node in $V_+ \setminus V'$ has most of its neighbors in $U_+$;
 hence every node outside of $V'$ will be correctly labelled.
 On the other hand,
 Proposition~\ref{prop:Ve-adjacent} shows that nodes in $V'$ are also correctly
 labelled, since none of them have any neighbors in $V_\epsilon$ (recalling
 that $V' \subset V_\epsilon$).
\end{proof}

\section{Necessary condition for strong consistency}

\label{sec:necessary-strong}

A classical fact in Bayesian statistics says that if we are asked to produce
a configuration $\hat \sigma$ from the graph $G$, then the algorithm
with the highest probability of success is the \emph{maximum
a posteriori} estimator, $\hat \sigma$, which is defined to be
any $\tau \in \{-1, 1\}^{V(G)}$ satisfying $\sum_u \tau_u = 0$
that maximizes $\Pr(G \mid \sigma = \tau)$.
(To see that this is the estimator with the highest probability of success,
note that every $\tau$ that maximizes $\Pr(G \mid \sigma = \tau)$ 
also maximizes $\Pr(\sigma = \tau \mid G)$; clearly, a $\tau$ that
maximizes the latter quantity is an optimal estimate.)
In order to prove that $P(n, p_n, q_n) = o(n^{-1})$ is necessary for
strong consistency, we relate the success probability of
$\hat \sigma$ to the existence of nodes with minorities.
Note that we say $v$ has a majority with respect to
$\tau$ if (assuming $p > q$) $\tau$ gives the same label to $v$ as
it does to most of $v$'s neighbors.

\begin{lemma}\label{lem:MAP-minority}
If there is a unique maximal $\hat \sigma$ then
with respect to $\hat \sigma$, there cannot be both a $+$-labelled node
with a minority and a $-$-labelled node with a minority.
\end{lemma}

\begin{proof}
For convenience, we will assume that $p > q$. The same
proof works for $p < q$, but one needs to remember that the definition
of ``majority'' and ``minority'' swap in that case
(Definition~\ref{def:majority}).

The probability of $G$ conditioned on the labelling $\tau$
may be written explicitly: if $A_\tau$ is the set of unordered pairs
$u \ne v$ with $\tau_u = \tau_v$ and $B_\tau$ is the set of unordered pairs
$u \ne v$ with $\tau_u \ne \tau_v$ then
\begin{align}
 \Pr(G \mid \sigma = \tau)
 &= p^{|E(G) \cap A_\tau|}
   q^{|E(G) \cap B_\tau|}
   (1-p)^{|A_\tau \setminus E(G)|}
   (1-q)^{|B_\tau \setminus E(G)|} \notag \\
 &= (1-p)^{|A_\tau|} (1-q)^{|B_\tau|}
  \left(\frac{p}{1-p}\right)^{|E(G) \cap A_\tau|}
  \left(\frac{q}{1-q}\right)^{|E(G) \cap B_\tau|}.
  \label{eq:G-given-tau}
\end{align}
Consider a labelling $\tau$.
Suppose that there exist nodes $u$ and $v$ with $\tau_u = +$
and $\tau_v = -$, and such that both $u$ and $v$ have minorities with respect
to $\tau$. We will show that $\tau$ cannot be the unique
maximizer of $\Pr(G \mid \sigma = \tau)$, which will establish the lemma.

Consider the labelling $\tau'$ that is identical to $\tau$
except that $\tau'_u = -$ and $\tau'_v = +$. The fact that $u$ and $v$
both had minorities with respect to $\tau$ implies that
\begin{align*}
 |E(G) \cap A_{\tau'}| &\ge |E(G) \cap A_\tau| \\
 |E(G) \cap B_{\tau'}| &\ge |E(G) \cap B_\tau|
\end{align*}
(note that equality is possible in the inequalities above if
$u$ and $v$ are neighbors). On the other hand, the number of $+$
and $-$ labels are the same for $\tau$ and $\tau'$; hence
$|A_\tau| = |A_{\tau'}|$ and $|B_\tau| = |B_{\tau'}|$.
Looking back at~\eqref{eq:G-given-tau}, therefore, we have
\[
 \Pr(G \mid \sigma = \tau) \le \Pr(G \mid \sigma = \tau').
\]
Hence, $\tau$ cannot be the unique maximizer of $\Pr(G \mid \sigma = \tau)$.
\end{proof}

In order to argue that $P(n, p_n, q_n) = o(n^{-1})$ is necessary
for strong consistency, we need to show that if
$P(n, p_n, q_n)$ is not $o(n^{-1})$ then $(G, \sigma) \sim \calG(2n, p_n, q_n)$
has a non-vanishing chance of containing nodes of both labels with minorities.

Suppose that $P(n, p_n, q_n)$ is not $o(n^{-1})$. By
Proposition~\ref{prop:majorities}, there is some $\epsilon > 0$ such
that for infinitely many $n$,
$\Pr(\exists u: \text{$u$ has a minority}) \ge \epsilon$. Since $+$-labelled
nodes and $-$-labelled nodes are symmetric, there are infinitely many $n$
such that
\begin{align*}
\Pr(\exists u: \sigma_u = + \text{ and $u$ has a minority}) &\ge \epsilon / 2 \\
\Pr(\exists v: \sigma_v = - \text{ and $u$ has a minority}) &\ge \epsilon / 2.
\end{align*}
By Harris's inequality~\cite{Harris:60},
the two events above are non-negatively correlated because both
of them are monotonic events with the same directions: both are monotonic
increasing in the edges between $+$-labelled and
$-$-labelled nodes and monotonic decreasing in the other edges.
Hence, there are infinitely many $n$ for which
\[
\Pr(\exists u, v: \sigma_u = +, \sigma_v = -, \text{ $u$ and $v$ have minorities})
\ge \epsilon^2 / 4.
\]

\section{Binomial approximations}\label{sec:binomial}

In this section, we collect various technical, but not particularly
enlightening, estimates for binomial variables.
Specifically, we prove Propositions~\ref{prop:explicit-sparse}
and~\ref{prop:explicit-dense}, which give explicit characterizations
of the condition $P(n, p_n, q_n) = o(n^{-1})$ in the sparse and dense case respectively,
and Proposition~\ref{prop:large-majority} and~\ref{prop:large-majority-sparse},
which give perturbative estimates for binomial probabilities.
Our main tools are
Bernstein's inequality, Stirling's approximation and Taylor expansion.

\subsection{Characterization of sparse strong consistency}

For simplicity, in this section we write $a = a_n, b = b_n$ and $c = a + b$.
If there is a constant $C > 0$ such that $C^{-1} f \le g \le C f$ then we
write $f \asymp g$. We recall that $a, b = \Theta(1)$ and that $p n = a \log n$
and $q n = b \log n$.
Let $X \sim \Binom(n, p)$ and $Y \sim \Binom(n, q)$.

We begin with a Poisson approximation to binomials.
\begin{lemma}\label{lem:poisson}
  If $Z = X + Y$ then for every $k \le 10 c \log n$,
  \[
      \Pr(Z = k) = (1 + o(1)) n^{-c} \frac{(c \log n)^k}{k!},
  \]
  where the sequence implicit in the $o(1)$ notation is independent of $n$ and $k$.
\end{lemma}

\begin{proof}[Proof (sketch)]
  By a direct computation, if $k \le 10 c \log n$ then
  \begin{align*}
    \Pr(X = k) = (1 + o(1)) n^{-a} \frac{(a \log n)^k}{k!} \\
    \Pr(Y = k) = (1 + o(1)) n^{-b} \frac{(b \log n)^k}{k!},
  \end{align*}
  and the sequences implicit in the $1 + o(1)$ notation may be taken to be
  independent of $k$. Finally, note that
  $\Pr(Z = k) = \sum_{\ell=0}^k \Pr(Y = \ell) \Pr(X = k - \ell)$.
\end{proof}

\begin{proof}[Proof of Proposition~\ref{prop:explicit-sparse}]
We first note that if $a - b \leq \epsilon = \epsilon(C)$ then strong consistency does
not hold. This follows because with constant probability we have that $X$
is less than its mean $a_n \log n$ and the probability that $Y$ is larger than
$a \log n$ is at least $n^{-1/2}$ if $\eps$ is a sufficiently small constant.

Without loss of generality, we may assume that $c \ge 1$. Indeed,
if $c < 1$ then the proposition is trivially true:
on the one hand $P(n, p_n, q_n) = \Omega(n^{-1})$
because $\Pr(X = 0)$ and $\Pr(Y = 0)$ are both $\Omega(n^{-1})$;
on the other hand, $(a + b - 2\sqrt{ab} - 1) \log n + \frac 12 \log \log n
\to -\infty$ because $a + b = c < 1$ and $\sqrt{ab}$ is bounded away from zero as $n \to \infty$.

Let $Z = X+Y$; then
\begin{align*}
\Pr(Y \geq X) &= \sum_{k = 0}^n \Pr(Z = k) \Pr(Y \geq X \mid Z = k)\\
&= \sum_{k = 0}^{10 c \log n} \Pr(Z = k) \Pr(Y \geq X \mid Z = k) + O(n^{-2}),
\end{align*}
where the second equality follows from the fact that $\Pr(Z \geq 10 c \log n) \leq O(n^{-2})$,
recalling that $c \geq 1$.

For a fixed $k \leq 10 c \log n$, we have that
\[
\Pr(Y \geq X \mid Z = k) = (1-o(1)) \Pr(\Binom(k,\eta) \geq k/2),
\]
where $\eta = \frac{b}{a+b} \le \frac 12 (1-\eps)$.
Recall that binomial tail probabilities decay exponentially fast; since $\eta \le \frac 12 (1-\eps)$,
$\Pr(\Binom(k,\eta) \ge k/2) \asymp \Pr(\Binom(k,\eta) = \lceil k/2 \rceil)$.Combining this
with Stirling's approximation we have
\[
\Pr(Y \geq X \mid Z = k) \asymp \frac{2^k}{\sqrt{k}} \eta^{k/2}(1-\eta)^{k/2} = \frac{2^k \theta^k}{\sqrt{k}},
\]
where $\theta = \sqrt{\eta(1-\eta)} = \frac{\sqrt{a b}}{a+b}$.
By Lemma~\ref{lem:poisson},
\[
\Pr(Z = k) = (1+o(1)) n^{-c} \frac{(c \log n)^k}{k!},
\]
and so Stirling's approximation for $k \ge 1$ gives
\[
\Pr(Z = k) \asymp \frac{n^{-c}}{\sqrt{k}} \frac{(c e \log n)^k}{k^k}
\]
Thus we get that
\begin{align*}
\Pr(Y \geq X) &= \Pr(Y = X = 0) + \sum_{k=1}^{10 c \log n} \Pr(Z = k) \Pr(Y \geq X \mid Z = k) +
O(n^{-2}) \\
& \asymp n^{-c} \left( 1+ \sum_{k=1}^{10 c \log n} \frac{(2 c e \theta \log n)^k}{k^{k+1}} \right),
\end{align*}
The analysis of the sum is standard, and we give a sketch. Defining $\ell(k)$ to be the
logarithm of the summand, we have
\[
\ell(k) = k \log (t \log n) - (k+1) \log k, \quad t = 2 c e \theta.
\]
Then
\[
\ell'(k) = \log (t \log n) - (1+1/k) - \log k, \quad \ell''(k) = -1/k(1+o(1)),
\]
and so the maximum is obtained around the value
\[
k^{\ast} = e^{-1} t \log n = 2 c \theta \log n.
\]
Moreover, the maximum value (up to a constant factor) of $\ell$ is
\[
\frac{(2 c e \theta \log n)^{k^{\ast}}}{k^\ast (2 c \theta \log n)^{k^{\ast}}} =
\frac{e^{k^{\ast}}}{k^{\ast}} \asymp \frac{n^{-c + 2c \theta}}{\log n} =
\frac{n^{2\sqrt{ab}}}{\log n}
\]
Since $\ell$ is approximately quadratic around its maximum and $\ell''(k^\ast) \asymp -1/\log n$,
we see that $\exp(\ell(k))$ varies by a constant factor on a window of length $\sqrt{\log n}$
around $k^\ast$, and then drops off geometrically fast beyond that window. Hence, the sum
is given (up to a constant) by $n^{2\sqrt{ab}} \log^{-1/2} n$
and so
\[
  \Pr(Y \ge X) \asymp
  \frac{n^{2\sqrt{ab} - (a+b)}}{\sqrt{\log n}}
\]
Thus $n \Pr(Y \geq X) \to 0$ if and only if
\[
(1+2 \sqrt{ab}-(a+b))\log n - \frac 12 \log \log n \to -\infty,
\]
as needed.
\end{proof}

\subsection{Characterization of dense strong consistency}
Our main tool for proving Proposition~\ref{prop:explicit-dense}
will be the following Local Central Limit Theorem. The proof is a standard
application of Stirling's approximation.
\begin{lemma}\label{lem:lclt}
Let $C > 0$ be an arbitrary constant and $Y \sim \Binom(n,q)$, where
\[
q = q_n = \omega\left(\frac{\log^3(n)}{n}\right), \quad q_n \leq \frac 23.
\]
Let $\sigma_q^2 = q(1-q)$ and let $\phi(x) = (2\pi)^{-1/2} e^{-x^2/2}$.
 Then for all integers $k$ such that $|k-nq| \leq C \sqrt{n \log n} \sigma_q$ it holds that
\[
\Pr(Y = k) = (1+o(1)) \frac{1}{\sqrt n \sigma_q}
\phi \left( \frac{k-nq}{\sqrt{n} \sigma_q} \right).
\]
Moreover,
\[
\Pr(Y = k) = (1+o(1)) \frac{1}{\sqrt n \sigma_q}
\phi \left( \frac{x-nq}{\sqrt{n} \sigma_q} \right),
\]
for every $k-1 \leq x \leq k+1$.
\end{lemma}

\begin{proof}
The second statement follows easily from the first one using the formula for $\phi$ and noting that
if $\delta \leq C \sqrt{n \log n} \sigma_q$ and $|\eps| \leq 1$ then
\[
\left(\frac{\delta + \eps}{\sigma_q \sqrt{n}}\right)^2 = \left( \frac{\delta}{\sigma_q \sqrt{n}} \right)^2 + o(1).
\]

To prove the first statement, we begin with Stirling's approximation.
Noting that  $k \to \infty$ as $n \to \infty$, we obtain:
\[
\Pr(Y = k) = \binom{n}{k} q^k(1-q)^{n-k} =
(1+o(1)) \frac{1}{\sqrt{2 \pi}} \sqrt{\frac{n}{k(n-k)}}
\left(\frac{n q}{k}\right)^k \left(\frac{n(1-q)}{n-k}\right)^{n-k}.
\]
We start by analyzing the term
\[
\sqrt{\frac{n}{k(n-k)}} = \frac{1}{\sqrt{n}} \sqrt{\frac{n}{k}} \sqrt{\frac{n}{n-k}}.
\]
Now
\[
k/n \in [q - C \frac{\sigma_q \sqrt{\log n}}{\sqrt{n}}, q+ C\frac{\sigma_q \sqrt{\log n}}{\sqrt{n}}]
\]
and since $q = \omega(n^{-1} \log^3 n)$ implies $\frac{\sigma_q \sqrt{\log n}}{\sqrt{n}} = o(q/\log n)$,
it follows that $n/k = (1+o(1/\log n)) \frac 1q$. Similarly,
$\frac{n}{n-k} = (1 + o(1/\log n)) \frac{1}{1-q}$ and so
\begin{equation}\label{eq:poly-term}
\sqrt{\frac{n}{k(n-k)}} = (1+o(1/\log n)) \frac{1}{\sigma_q \sqrt{n}}.
\end{equation}
Next, we use Taylor expansion around $nq = k$. The first-order term vanishes
and we have
\begin{multline}
\log \left( \left(\frac{n q}{k}\right)^k \left(\frac{n(1-q)}{n-k}\right)^{n-k} \right) \\
\begin{aligned}
 &= -\frac{1}{2} (k-nq)^2 \left(\frac{1}{k} + \frac{1}{n-k}\right)
 + O(|nq-k|^3)\left(\frac{1}{k^2}+\frac{1}{(n-k)^2}\right)
\\ &= - \frac{n}{2k(n-k)} (k-nq)^2 + o(1),
\label{eq:exp-term}
\end{aligned}
\end{multline}
where the last equality uses the fact that
\[
\frac{(nq-k)^3}{\min\{k^2,(n-k)^2\}} \to 0,
\]
which follows from the assumption that $q = \omega(n^{-1}\log^3(n))$.
Now, from~\eqref{eq:poly-term} we have
$\frac{n}{k(n-k)} = (1 + o(1/\log n)) \frac{1}{n \sigma_q^2}$.
Since $(k-nq)^2 = O(\sigma_q^2 n \log n)$, we have
\[
\frac{n}{k(n-k)} (k-nq)^2 = \frac{(k-nq)^2}{n \sigma_q^2} + o(1).
\]
Going back to~\eqref{eq:exp-term}, we have
\[
\log \left( \left(\frac{n q}{k}\right)^k \left(\frac{n(1-q)}{n-k}\right)^{n-k} \right)
= - \frac{(k-nq)^2}{2n\sigma_q^2} + o(1).
\]
The proof follows by combining this with~\eqref{eq:poly-term}
and Stirling's approximation for $\Pr(Y = k)$.
\end{proof}

\begin{proof}[Proof of Proposition~\ref{prop:explicit-dense}]
The second and third conditions are clearly equivalent; we will show
the equivalence of the first two.

Bernstein's inequality implies that
\[
\Pr(|Y-\E Y| \geq 4 \sqrt{n \log n} \sigma_q) = o(n^{-1}),
\Pr(|X-\E X| \geq 4 \sqrt{n \log n} \sigma_p) = o(n^{-1}).
\]
So writing $b_q = 5 \sqrt{n \log n} \sigma_q$ and $b_p = 5 \sqrt{n \log n} \sigma_p$ we have:
\[
\Pr(Y \geq X) =
\sum_{k = \lfloor np - b_p \rfloor}^{\lceil np + b_p \rceil} \sum_{\ell = \lfloor nq - b_q\rfloor}^{\lceil nq + b_q \rceil} 1_{\{k \leq \ell\}}
\Pr(X = k) \Pr(Y = \ell) +o(n^{-1})
\]
Using Lemma~\ref{lem:lclt} for every $k,\ell$ in the range above we have:
\[
\Pr(X = k) \Pr(Y = \ell) = (1+o(1)) \frac{1}{n \sigma_p \sigma_q}
\int_{\Delta(k,\ell)} \phi \left( \frac{y-nq}{\sqrt{n} \sigma_q} \right)
\phi \left( \frac{x-np}{\sqrt{n} \sigma_p} \right) dx dy,
\]
where $\Delta(k,\ell) = (k,\ell) + \Delta$ where
\[
\Delta = \{ (x,y) : 0 \leq y \leq 1, \ y-1 \le x \le y \}
\]
is a parallelogram of unit area. (In applying Lemma~\ref{lem:lclt}
  note that $(x, y) \in \Delta(k, \ell)$ implies that
$|x - k| \le 1$ and $|y - \ell| \le 1$.)
Thus
\[
\Pr(Y \geq X) =
(1+o(1)) \int_{np - b_p}^{np + b_p} \int_{nq - b_p}^{nq + b_q}
1{\{x \le y\}}
\phi \left( \frac{y-nq}{\sqrt{n} \sigma_q} \right)
\phi \left( \frac{x-np}{\sqrt{n} \sigma_p} \right) dy dx + o(n^{-1}),
\]
where we use the fact that the difference between the union of $\Delta(k,\ell)$ and
the integration region above is contained in the set where either
$|y-nq| \geq 4 \sqrt{n \log n} \sigma_q$ or $|x-np| \geq 4 \sqrt{n \log n} \sigma_p$.
Changing variables we see that the last expression is nothing but
\[
\Pr\left( |M| \leq 5 \sqrt{n \log n}, \ |N| \leq 5 \sqrt{n \log n}, \ 
\sigma_q M \geq \sqrt{n}(p-q) + \sigma_p N\right),
\]
Where $M,N \sim \normal(0,1)$ are independent. The proof follows.
\end{proof}

\subsection{Perturbation estimates for dense binomials}

The main approximation that we use to prove
Proposition~\ref{prop:large-majority} is the following:
\begin{lemma}\label{lem:ratio}
  If $X \sim \Binom(m, p)$ then for any $k$ and $\ell$,
 \[
  \log \frac{\Pr(X = k + \ell)}{\Pr(X = k)} \le
  \ell \log \frac{mp}{k+1} + \ell \log \frac{m-k}{m-mp}.
 \]
\end{lemma}

\begin{proof}
 We compute
 \begin{align*}
  \log \frac {\Pr(X = k + \ell)}{\Pr(X = k)}
  &= \log \frac {\binom{m}{k+\ell} p^\ell} {\binom mk (1-p)^\ell} \\
  &= \ell \log \frac{p}{1-p} + \sum_{i=1}^\ell (\log(m-k-i+1) - \log(k+i)) \\
  &\le \ell \log \frac{p}{1-p} + \ell \log(m-k) - \ell \log(k+1) \\
  &= \ell \log \frac{mp}{k+1} + \ell \log \frac{m-k}{m-mp}.
  \qedhere
 \end{align*}
\end{proof}

\begin{proof}[Proof of Proposition~\ref{prop:large-majority}]
 Fix $\ell$ with $1 \le \ell \le \sqrt{mp \log m}$.
 We will focus on the proof of~\eqref{eq:large-majority-neg},
 since the proof of~\eqref{eq:large-majority-pos} is analogous.
 We may write
 \[
  \Pr(Y \ge X - \ell)
  = \sum_{k=-\ell}^m \Pr(Y \ge k) \Pr(X = k + \ell).
 \]
 Now, Bernstein's inequality implies that by incurring
 a cost of $2m^{-2}$, we may restrict the sum to those $k$ for which
 $mp - 3\sqrt{mp \log m} \le k + \ell \le mp + 3\sqrt{mp \log m}$.
 Since $\ell \le \sqrt{mp \log m}$, it suffices to take
 $mp - 4\sqrt{mp \log m} \le k \le mp + 4\sqrt{mp \log m}$. Hence,
 \begin{equation}\label{eq:constant-majority-1}
  \Pr(Y \ge X - \ell)
  \le \sum_{k=\lfloor mp - 4\sqrt{mp \log m}\rfloor}
	  ^{\lceil mp + 4\sqrt{mp \log m}\rceil}
	  \Pr(Y \ge k) \Pr(X = k + \ell) + 2m^{-2}.
 \end{equation}
 Now, under the assumption $mp \ge 64 \log m$, we have
 $mp - 4\sqrt{mp \log m} \ge mp/2$ and
 $mp + 4\sqrt{mp \log m} \le 3mp/2$.
 Consider the first term in the upper bound of Lemma~\ref{lem:ratio}:
 \begin{equation}\label{eq:ratio-term-1}
  \log \frac{mp}{k+1}
  \le \frac{|k+1-mp|}{\min\{k+1,mp\}}
  \le 16 \sqrt{\frac{\log m}{mp}}
 \end{equation}
 where the last inequality used $|k - mp| \le 4 \sqrt{mp \log m}$ and
 $k \ge mp/2$. The other term in the upper bound of Lemma~\ref{lem:ratio} is similar:
 \begin{equation}\label{eq:ratio-term-2}
     \log \frac{m-k}{m-mp} \le \frac{|k-mp|}{\min\{m-mp, m-k\}}
     \le C \sqrt{\frac{\log m}{mp}}
 \end{equation}
 for sufficiently large $m$,
 where the second inequality follows by lower-bounding both terms in the
 denominator: $p \le 2/3$ implies
 $m - mp \ge 2mp$ and $k \le mp + 4 \sqrt{mp \log m}$ implies
 $m - k \ge cmp$  for some $c > 0$ and sufficiently large $m$
 (this follows by considering the cases $p \in [2^{-10}, 2/3]$ and
 $p \in [64 m^{-1} \log m, 2^{-10}]$ separately).
 Combining~\eqref{eq:ratio-term-1} and~\eqref{eq:ratio-term-2}
 with Lemma~\ref{lem:ratio}, we obtain
 \begin{equation}\label{eq:compare-probs}
  \log \frac{\Pr(X = k+\ell)}{\Pr(X = k)}
  \le C \ell \sqrt{\frac{\log m}{mp}}.
 \end{equation}
 Applying this to~\eqref{eq:constant-majority-1}, we have
 \begin{align*}
  \Pr(Y \ge X - \ell)
  &\le \exp\left(C \ell \sqrt{\frac{\log m}{mp}} \right)
      \sum_{k=\lfloor mp - 4 \sqrt{mp \log m}\rfloor}
	  ^{\lceil mp + 4 \sqrt{mp \log m}\rceil}
	  \Pr(Y \ge k) \Pr(X = k) + 2m^{-2} \\
  &\le \Pr(Y \ge X) \exp\left(C \ell \sqrt{\frac{\log m}{mp}}\right) + 2m^{-2}.
 \end{align*}

 The lower bound (i.e.~\eqref{eq:large-majority-pos})
 is essentially the same, and we give only a sketch: we write
 \begin{align*}
   \Pr(Y \ge X + \ell) \ge \sum_{k=\lfloor mp - 4 \sqrt{mp \log m}\rfloor}^{\lceil mp + 4 \sqrt{mp \log m}\rceil}
   \Pr(Y \ge k + \ell) \Pr(X = k).
 \end{align*}
 We then use~\eqref{eq:compare-probs} to compare $\Pr(X = k)$ with
 $\Pr(X = k + \ell)$. This leaves us with a sum over $k \in mp \pm 4 \sqrt{mp \log m}$,
 which we compare with the full sum using Bernstein's inequality (picking up
 an additive $2m^{-2}$ term).
\end{proof}

\subsection{Perturbation estimates for sparse binomials}

The sparse case needs a slightly different argument and
has slightly worse bounds. We have the following analogue
of Lemma~\ref{lem:ratio}:

\begin{lemma}\label{lem:ratio-sparse}
 If $mp \le 128 \log m$ and $k = o(m)$ then for sufficiently large $m$ and any $\ell \ge 1$,
 \[
  \log \frac {\Pr(X = k + \ell)}{\Pr(X = k)}
  \le \ell \log \frac{mp}{\ell} + 2\ell
 \]
\end{lemma}

\begin{proof}
 As in the proof of Lemma~\ref{lem:ratio}, we compute
 \begin{align*}
  \log \frac {\Pr(X = k + \ell)}{\Pr(X = k)}
  &= \ell \log \frac{p}{1-p} + \sum_{i=1}^\ell (\log(m-k-i+1) - \log(k+i)) \\
  &\le \ell \log \frac{p}{1-p} + \ell \log(m-k) - \sum_{i=1}^\ell \log(k+i).
 \end{align*}
 This time, we will use a sharper bound on the sum:
 since the logarithm is an increasing function,
 \begin{align*}
  \sum_{i=1}^\ell \log(k+i)
  &\ge \int_k^{k+\ell} \log(x) \, dx \\
  &= (k+\ell) \log(k+\ell) - (k+\ell) - k \log k + k \\
  &\ge \ell \log(k+\ell) - \ell.
 \end{align*}
 Hence, we obtain
 \[
  \log \frac {\Pr(X = k + \ell)}{\Pr(X = k)}
  \le \ell \log \frac{mp}{k+\ell} + \ell \log\frac{m-k}{m-mp} + \ell.
 \]
 Since $k$ and $mp$ are $o(m)$, $\log((m-k)/(m-mp)) = o(1)$,
 and so
 \[
  \log \frac {\Pr(X = k + \ell)}{\Pr(X = k)}
  \le \ell \log \frac{mp}{\ell} + 2\ell
 \]
 for sufficiently large $m$.
\end{proof}

\begin{proof}[Proof of Proposition~\ref{prop:large-majority-sparse}]
 This proof is similar to the proof of Proposition~\ref{prop:large-majority},
 but with Lemma~\ref{lem:ratio-sparse} instead of Lemma~\ref{lem:ratio}
 and some slightly different truncations:
 we write
 \[
   \Pr(Y \ge X - \ell) = \Pr(X \le \ell - 1)
   + \sum_{k=0}^m \Pr(Y \ge k) \Pr(X = k + \ell)
 \]
 By Bernstein's inequality, we may truncate the sum at $\sqrt m$
 at the cost of an additive $e^{-c \sqrt m}$ term.
 We apply the inequality
 \[
     \frac{\Pr(X = k + \ell)}{\Pr(X = k)} 
     \le \left(\frac{e^2 mp}{\ell}\right)^\ell
     \le \left(\frac{C \log m}{\ell}\right)^\ell
 \]
 (which follows from Lemma~\ref{lem:ratio-sparse}) to each term
 in the sum, yielding
 \[
   \sum_{k=0}^m \Pr(Y \ge k) \Pr(X = k + \ell)
   \le \left(\frac {C \log m}{\ell}\right)^\ell \Pr(Y \ge X) + e^{-c\sqrt m}.
 \]

 We may also apply
 Lemma~\ref{lem:ratio-sparse} to bound the term $\Pr(X \le \ell - 1)$, using
 \begin{align*}
   \Pr(X \le \ell - 1)
   &= \sum_{s=0}^{\ell - 1} \Pr(X = s) \\
   &\le \sum_{s=0}^{\ell - 1} \left(\frac{C \log m}{s}\right)^s \Pr(X = 0) \\
   &\le \ell \left(\frac{C \log m}{\ell}\right)^{\ell} \Pr(X = 0) \\
   &\le \left(\frac{C \log m}{\ell}\right)^{C\ell} \Pr(X = 0),
 \end{align*}
 where the second inequality follows (assuming $C \ge e$)
 because $(ey/x)^x$ is an increasing function of $x$ for $x \le y$.
 Putting everything together,
 \[
   \Pr(Y \ge X - \ell) \le 
    \left(\frac{C \log m}{\ell}\right)^{C\ell} \Pr(X = 0)
    + \left(\frac{C \log m}{\ell}\right)^{\ell} \Pr(Y \ge X)
    + e^{-c\sqrt m}.
 \]
 Finally, note that $\Pr(X = 0) \le \Pr(Y \ge X)$ so that the first
 two terms above may be combined at the cost of increasing $C$. For
 the additive term $e^{-c \sqrt m}$, note that $mp \le 128 \log m$
 implies that $\Pr(Y \ge X) \ge \Pr(X = 0) = \Omega(n^{-\alpha})$ for some
 constant $\alpha$, and so $e^{-c \sqrt m}$ may also be absorbed into
 the main term at the cost of increasing $C$.
\end{proof}

\section{Erratum}

The published version of this paper contained a mistake; we are grateful to Jan van Waaij for
pointing it out.

The statement of Lemma~\ref{lem:MAP-minority} is incorrect; the error in the proof was introduced in the inequality
\[
    |E(G) \cap A_{\tau'}| \ge |E(G) \cap A_\tau|,
\]
which does not hold under the assumption of Lemma~\ref{lem:MAP-minority}.
To formulate a correct version, we introduce the notion of a strict minority:
\begin{definition}
    Given a labelled graph $(G, \sigma)$, we say that $v$ has a strict minority if either
    \[
        p > q \text{ and } \#\{u \sim v: \sigma_u = \sigma_v\} < \#\{u \sim v: \sigma_u \ne \sigma_v\} 
    \]
    or
    \[
        p < q \text{ and } \#\{u \sim v: \sigma_u \ne \sigma_v\} < \#\{u \sim v: \sigma_u = \sigma_v\} .
    \]
\end{definition}

Here is a corrected version of Lemma~\ref{lem:MAP-minority} (using the notation of Lemma~\ref{lem:MAP-minority}):

\begin{lemma}\label{lem:MAP-minority-correct}
    If there is a unique maximal $\hat \sigma$ then with respect to $\hat \sigma$ then there cannot
    be both a $+$-labelled node $u$ and a $-$-labelled node $v$ such that either
    \begin{enumerate}
        \item $u$ and $v$ both have strict minorities, or
        \item $u$ and $v$ are non-adjacent and both have minorities.
    \end{enumerate}
\end{lemma}

\begin{proof}
    The proof of Lemma~\ref{lem:MAP-minority-correct} is essentially the same as the proof
    of Lemma~\ref{lem:MAP-minority}, except that the strengthened assumption
    means that the problematic inequality is now true.
    As before, assume that $p > q$, let $u$ and $v$ be any nodes with $\tau_u = +$ and $\tau_v = -$,
    let $A_\tau = \{\{u, v\}: \tau_u = \tau_v\}$
    and let $\tau'$ be the labelling obtained from $\tau$ by swapping the labels of $u$ and $v$.
    We need to show that under either of the two conditions in Lemma~\ref{lem:MAP-minority-correct},
    $|E(G) \cap A_{\tau'}| \ge |E(G) \cap A_\tau|$.

    The sets $E(G) \cap A_{\tau'}$ and $E(G) \cap A_\tau$
    differ only among edges that are incident to either $u$ or $v$, so it suffices to consider such edges, of which there are five types:
    \begin{enumerate}[a)]
        \item if $w \sim u$ has $\tau_w = +$ then $\{u, w\} \in A_\tau$ but not $A_{\tau'}$;
        \item if $w \sim u$, $w \ne v$ has $\tau_w = -$ then $\{u, w\} \in A_{\tau'}$ but not $A_\tau$;
        \item if $w \sim v$ has $\tau_w = -$ then $\{v, w\} \in A_\tau$ but not $A_{\tau'}$;
        \item if $w \sim v$, $w \ne u$ has $\tau_w = +$ then $\{v, w\} \in A_{\tau'}$ but not $A_\tau$;
        \item $\{u, v\}$ belongs to neither $A_\tau$ nor $A_{\tau'}$.
    \end{enumerate}
    Let $N_a$ through $N_e$ be the number of edges of $G$ corresponding to each of the types above.
    Then $|E(G) \cap A_{\tau'}| - |E(G) \cap A_\tau| = N_b + N_d - N_a - N_c$.
    Note that $N_e$ is either zero or one, and it is one if and only if $\{u, v\} \in E(G)$, and 
    note also that $u$ has a minority if and only if $N_a \le N_b + N_e$, while $u$
    has a strict minority if and only if $N_a \le N_b + N_e - 1$ (and similarly for $v$).
    Hence, if $u$ and $v$ both have strict minorities then
    \[
        |E(G) \cap A_{\tau'}| - |E(G) \cap A_\tau| = N_b + N_d - N_a - N_c \ge 2 - 2 N_e \ge 0,
    \]
    while if $u$ and $v$ both have minorities and are non-adjacent then
    \[
        |E(G) \cap A_{\tau'}| - |E(G) \cap A_\tau| = N_b + N_d - N_a - N_c \ge - 2 N_e = 0.
    \]
    Hence, in either case we have established that $|E(G) \cap A_{\tau'}| \ge |E(G) \cap A_\tau|$.

    Finally, note that if $B_\tau = \{\{u,v\}: \tau_u \ne \tau_v\}$ then $|E \cap B_\tau| = |E| - |E \cap A_\tau|$ and so~\eqref{eq:G-given-tau}
    implies that
    \[
        \Pr(G \mid \sigma = \tau) = (1-p)^{|A_\tau|} q^{|E(G)|} (1-q)^{|B_\tau| - |E(G)|} \left(
            \frac{\quad\frac{p}{1-p}\quad}{\frac{q}{1-q}}
        \right)^{|E(G) \cap A_\tau|}.
    \]
    If it were possible to increase $|E(G) \cap A_\tau|$ while maintaining $|E(G)|$, $|A_\tau|$, and $|B_\tau|$, $\tau$
    could not have been the unique maximum a posteriori estimator.
\end{proof}

Since the incorrect Lemma~\ref{lem:MAP-minority} was used to prove that
$P(n, p_n, q_n) = o(n^{-1})$ is necessary for strong consistency, we will now show
how Lemma~\ref{lem:MAP-minority-correct} can be used for the same purpose.
So, for the rest of the section we fix some $\epsilon > 0$ and assume (after passing
to a subsequence of $n$, if necessary) that $P(n, p_n, q_n) \ge \epsilon n^{-1}$.
We will divide the proof into a sparse case and a dense case. In the sparse case,
we show that there is a pair of non-adjacent minorities:

\begin{lemma}\label{lem:erratum-sparse}
    If $P(n, p_n, q_n) \ge \epsilon n^{-1}$ and $n p_n \le 64 \log n$ for
    infinitely many $n$ then with asymptotically positive probability there is
    a non-adjacent pair $u$, $v$ of nodes such that $\sigma_u = +$, $\sigma_v =
    -$, and $u$ and $v$ have minorities.
\end{lemma}

\begin{proof}
    For any set $S = S_n$ of at least $n/4$ vertices, let $N_S$ be the number
    of nodes in $S$ with a minority. The proof of Proposition~\ref{prop:majorities}
    shows that if $\alpha = \Pr(u \text{ is a minority}) \ge \epsilon n^{-1}$
    for all $n$ then $\E N_S = \alpha |S| \ge \epsilon / 4$ and $\Var(N_S) \le C \epsilon$;
    it follows from the Paley-Zygmund inequality that there is some $\delta$ such
    that for any $S = S_n$ with $|S| \ge n/4$,
    \[
        \Pr(\exists u \in S: u \text{ is a minority}) \ge \delta
    \]
    for all $n$.

    We will divide $\{u: \sigma_u = +\}$ into three sets $S_{1,+}, S_{2,+}$,
    and $S_{3,+}$, each of size at least $n/4$; similarly, we divide $\{u: \sigma_u = -\}$
    into $S_{1,-}$, $S_{2,-}$, and $S_{3,-}$. For each of these six sets $S_{i,j}$,
    $\Pr(\exists u \in S_{i,j}: u \text{ is a minority}) \ge \delta$.
    Next, note that the event
    that $u$ has a minority is (in the sense of Harris~\cite{Harris:60})
    monotone increasing in the edges between $+$-labelled and $-$-labelled nodes
    and monotone decreasing in the other edges. It follows from Harris's
    inequality~\cite{Harris:60} that any such events are non-negatively correlated.
    In particular, with probability at least $\delta^6$, every $S_{i,j}$ contains a
    node with a minority ($u_{i,j}$, say).
    
    We will complete the proof by showing that a.a.s.\ it is not the case that
    every $u_{i,+}$ is connected to every $u_{k,-}$.  Indeed, if every
    $u_{i,+}$ is connected to every $u_{k,-}$ then the graph $G$ contains a
    subgraph isomorphic to $K_{3,3}$ (the complete bipartite graph).  However,
    the random graph $G$ is stochastically dominated by the Erd\H{o}s-R\'enyi
    graph $\calG(n,64 n^{-1} \log n)$, and it is well-known (for example, by
    the first moment method) that such a graph a.a.s.\ does not contain a copy
    of $K_{3,3}$.
\end{proof}

To complete the proof we consider the dense case, where we prove that there is a pair of strict minorities:

\begin{lemma}\label{lem:erratum-dense}
    If $P(n, p_n, q_n) \ge \epsilon n^{-1}$ and $np_n \ge 64 \log n$ infinitely
    often then with asymptotically positive probability there are a pair $u$
    and $v$ with opposite labels and strict minorities.
\end{lemma}

\begin{proof}
    Fix a node $u$ with label $+$. Let $X$ and $Y$ be the number of
    $+$-neighbors and $-$-neighbors of $u$ respectively. The event $\{Y \ge
    X\}$ is the event that $u$ has a minority, and the event $\{Y \ge X + 1\}$
    is the event that $u$ has a strict minority.
    By~\eqref{eq:large-majority-pos}, the probability that $u$ has a strict
    minority is at least $\delta n^{-1}$, for some $\delta$ depending on
    $\epsilon$.  It follows that if $N_+$ is the number of $+$-labelled
    vertices with strict minorities, then $\E N_+ \ge \delta$.
    
    In order to
    prove that $\Pr(N_+ \ge 1)$ is bounded away from zero, it suffices to prove that
    $\Var(N_+) \le C (\E N_+)^2$ for some constant $C$. The proof of this is
    essentially the same as the proof of Proposition~\ref{prop:majorities},
    except that we need to consider strict minorities instead of non-strict
    minorities. (Also, we need only consider the dense case.) Since this
    is very similar to the existing argument, we will only give a sketch.
    The key is to prove that the events $\{u \text{ has a strict minority}\}$
    and $\{v \text{ has a strict minority}\}$ are approximately independent.
    If we define $\tilde P(m, n, p, q)$ by
    \[
        \tilde P(m, n, p, q) = \Pr(Y > X) \text{ for $X \sim \Binom(m, p)$, $Y \sim \Binom(n, q)$}
    \]
    then this amounts to proving an analogue of Lemma~\ref{lem:one-less} with
    $\tilde P$ instead of $P$. Since we are interested only in the dense case,
    this requires an analogue of~\eqref{eq:large-majority-pos} with strict inequalities.
    But this just follows from the existing formulation of~\eqref{eq:large-majority-pos},
    since
    \begin{align*}
        \Pr(Y > X + \ell)
        &= \Pr(Y \ge X + \ell + 1) \\
        &\ge \Pr(Y \ge X) e^{\left(-C (\ell+1) \sqrt{\frac{\log m}{mp}}\right)} - 2 m^{-2} \\
        &\ge \Pr(Y > X) e^{\left(-C (\ell+1) \sqrt{\frac{\log m}{mp}}\right)} - 2 m^{-2}
    \end{align*}
    and the difference between $\ell+1$ and $\ell$ can be absorbed into the constant $C$.

    Once we have established that $\Pr(N_+ \ge 1)$ is bounded away from zero,
    it follows by symmetry that $\Pr(N_- \ge 1)$ is bounded away from zero (where $N_-$ is the number of $-$-labelled nodes with a minority). As in the proof of Lemma~\ref{lem:erratum-sparse}, Harris's inequality implies that $\Pr(N_+ \ge 1 \text{ and } N_- \ge 1) \ge \Pr(N_+ \ge 1) \Pr(N_- \ge 1)$, and then it follows that with asymptotically positive probability there are strict minorities with both $+$ and $-$ labels.
\end{proof}

Finally, the proof that $P(n, p_n, q_n) = o(n^{-1})$ is necessary for strong
consistency follows by combining Lemma~\ref{lem:MAP-minority-correct} with Lemma~\ref{lem:erratum-sparse}
in the sparse case, or with Lemma~\ref{lem:erratum-dense} in the dense case.

\bibliography{all,block-model}

\begin{thebibliography}{10}

\bibitem{AbBaHa:14}
E.~Abbe, A.~S. Bandeira, and G.~Hall.
\newblock Exact recovery in the stochastic block model.
\newblock Arxiv 1405.3267.

\bibitem{ACBL:13}
Arash~A. Amini, Aiyou Chen, Peter~J. Bickel, and Elizaveta Levina.
\newblock Pseudo-likelihood methods for community detection in large sparse
  networks.
\newblock {\em The Annals of Statistics}, 41(4):2097--2122, 08 2013.

\bibitem{BC:09}
P.J. Bickel and A.~Chen.
\newblock A nonparametric view of network models and {N}ewman-{G}irvan and
  other modularities.
\newblock {\em Proceedings of the National Academy of Sciences},
  106(50):21068--21073, 2009.

\bibitem{B:87}
R.B. Boppana.
\newblock Eigenvalues and graph bisection: An average-case analysis.
\newblock In {\em 28th Annual Symposium on Foundations of Computer Science},
  pages 280--285. IEEE, 1987.

\bibitem{BCLS:87}
T.N. Bui, S.~Chaudhuri, F.T. Leighton, and M.~Sipser.
\newblock Graph bisection algorithms with good average case behavior.
\newblock {\em Combinatorica}, 7(2):171--191, 1987.

\bibitem{CI:01}
T.~Carson and R.~Impagliazzo.
\newblock Hill-climbing finds random planted bisections.
\newblock In {\em Twelfth annual ACM-SIAM symposium on Discrete algorithms},
  pages 903--909. Society for Industrial and Applied Mathematics, 2001.

\bibitem{CO:10}
A.~Coja-Oghlan.
\newblock Graph partitioning via adaptive spectral techniques.
\newblock {\em Combinatorics, Probability and Computing}, 19(02):227--284,
  2010.

\bibitem{CK:01}
A.~Condon and R.M. Karp.
\newblock Algorithms for graph partitioning on the planted partition model.
\newblock {\em Random Structures and Algorithms}, 18(2):116--140, 2001.

\bibitem{DF:89}
M.E. Dyer and A.M. Frieze.
\newblock The solution of some random {NP}-hard problems in polynomial expected
  time.
\newblock {\em Journal of Algorithms}, 10(4):451--489, 1989.

\bibitem{ErdosRenyi:61}
Paul Erd{\H{o}}s and Alfr{\'e}d R{\'e}nyi.
\newblock On the strength of connectedness of a random graph.
\newblock {\em Acta Mathematica Hungarica}, 12(1):261--267, 1961.

\bibitem{Harris:60}
Theodore~E. Harris.
\newblock A lower bound for the critical probability in a certain percolation
  process.
\newblock In {\em Mathematical Proceedings of the Cambridge Philosophical
  Society}, volume~56, pages 13--20. Cambridge Univ. Press, 1960.

\bibitem{HLL:83}
P.W. Holland, K.B. Laskey, and S.~Leinhardt.
\newblock Stochastic blockmodels: First steps.
\newblock {\em Social Networks}, 5(2):109 -- 137, 1983.

\bibitem{JS:98}
M.~Jerrum and G.B. Sorkin.
\newblock The {M}etropolis algorithm for graph bisection.
\newblock {\em Discrete Applied Mathematics}, 82(1-3):155--175, 1998.

\bibitem{Karp:72}
R.~Karp.
\newblock Reducibility among combinatorial problems.
\newblock In R.~Miller and J.~Thatcher, editors, {\em Complexity of Computer
  Computations}, pages 85--103. Plenum Press, 1972.

\bibitem{KomlosSzemeredi:83}
J{\'a}nos Koml{\'o}s and Endre Szemer{\'e}di.
\newblock Limit distribution for the existence of hamiltonian cycles in a
  random graph.
\newblock {\em Discrete Mathematics}, 43(1):55--63, 1983.

\bibitem{KumarKannan:10}
Amit Kumar and Ravindran Kannan.
\newblock Clustering with spectral norm and the k-means algorithm.
\newblock In {\em Foundations of Computer Science (FOCS), 2010 51st Annual IEEE
  Symposium on}, pages 299--308. IEEE, 2010.

\bibitem{MaMaVi:12}
Konstantin Makarychev, Yury Makarychev, and Aravindan Vijayaraghavan.
\newblock Approximation algorithms for semi-random partitioning problems.
\newblock In {\em Proceedings of the forty-fourth annual ACM symposium on
  Theory of computing}, pages 367--384. ACM, 2012.

\bibitem{MaMaVi:14}
Konstantin Makarychev, Yury Makarychev, and Aravindan Vijayaraghavan.
\newblock Constant factor approximation for balanced cut in the pie model.
\newblock In {\em Proceedings of the 46th Annual ACM Symposium on Theory of
  Computing}, pages 41--49. ACM, 2014.

\bibitem{Massoulie:13}
Laurent Massouli{\'e}.
\newblock Community detection thresholds and the weak ramanujan property.
\newblock In {\em Proceedings of the 46th Annual ACM Symposium on Theory of
  Computing}, pages 694--703. ACM, 2014.

\bibitem{M:01}
F.~McSherry.
\newblock Spectral partitioning of random graphs.
\newblock In {\em 42nd IEEE Symposium on Foundations of Computer Science},
  pages 529--537. IEEE, 2001.

\bibitem{MoNeSl:14b}
E.~Mossel, J.~Neeman, and A.~Sly.
\newblock Belief propagation, robust reconstruction, and optimal recovery of
  block models (extended abstract).
\newblock {\em JMLR Workshop and Conference Proceedings (COLT proceedings)},
  35:1--35, 2014.
\newblock Winner of best paper award at COLT 2014.

\bibitem{MoNeSl:13}
E.~Mossel, J.~Neeman, and A.~Sly.
\newblock Stochastic block models and reconstruction.
\newblock {\em Probability Theory and Related Fields}, 2014.
\newblock (to appear).

\bibitem{MoNeSl:14}
Elchanan Mossel, Joe Neeman, and Allan Sly.
\newblock A proof of the block model threshold conjecture.
\newblock (submitted to Combinatorica), 2014.

\bibitem{NadakuditiNewman:12}
Raj~Rao Nadakuditi and Mark~EJ Newman.
\newblock Graph spectra and the detectability of community structure in
  networks.
\newblock {\em Physical Review Letters}, 108(18):188701, 2012.

\bibitem{Seginer:00}
Yoav Seginer.
\newblock The expected norm of random matrices.
\newblock {\em Combinatorics, Probability and Computing}, 9:149--166, 3 2000.

\bibitem{Vu:07}
Van~H. Vu.
\newblock Spectral norm of random matrices.
\newblock {\em Combinatorica}, 27(6):721--736, 2007.

\bibitem{YunProutiere:14}
Se-Young Yun and Alexandre Proutiere.
\newblock Community detection via random and adaptive sampling.
\newblock {\em arXiv preprint \texttt{arXiv:1402.3072}}, 2014.

\end{thebibliography}
\bibliographystyle{plain}
\end{document}